\documentclass[12pt]{amsart}
\usepackage[utf8]{inputenc}
\usepackage{mathrsfs}
\usepackage{amssymb}
\usepackage{bm}
\usepackage{graphicx}
\usepackage[centertags]{amsmath}
\usepackage{amsfonts}
\usepackage{amsthm}
\usepackage{enumerate}
\usepackage{tocvsec2}
\usepackage{xcolor}
\usepackage[margin=1in]{geometry}

\newtheorem{theorem}{Theorem}[section]
\newtheorem*{theorem*}{Main Theorem}
\newtheorem*{corollary*}{Main Corollary}

\newtheorem{corollary}[theorem]{Corollary}
\newtheorem{lemma}[theorem]{Lemma}
\newtheorem{rem}[theorem]{Remark}

\newtheorem{proposition}[theorem]{Proposition}

\theoremstyle{definition}

\newenvironment{remark}[1][Remark]{\begin{trivlist}
\item[\hskip \labelsep {\bfseries #1}]}{\end{trivlist}}

\newcommand{\rr}{\mathbb{R}}
\newcommand{\nn}{\mathbb{N}}

\newcommand{\ee}{\varepsilon}

\newcommand{\fff}{\mathcal{F}}

\newcommand{\hhh}{\mathcal{H}}

\newcommand{\supp}{\mathrm{supp}}

\newcommand{\cat}{^\smallfrown}
\newcommand{\ord}{\textbf{Ord}}

\begin{document}
\title[Power type $\xi$-AUS norms]{Power type $\xi$-asymptotically uniformly smooth norms}
\author{R.M. Causey}
\address{Department of Mathematics, Miami University, Oxford, OH 45056, USA
} \email{causeyrm@miamioh.edu}

\begin{abstract} We extend a precise renorming result of Godefroy, Kalton, and Lancien regarding asymptotically uniformly smooth norms of separable Banach spaces with Szlenk index $\omega$.  For every ordinal $\xi$, we characterize the operators, and therefore the Banach spaces, which admit a $\xi$-asymptotically uniformly smooth norm with power type modulus and compute for those operators the best possible exponent in terms of the values of $Sz_\xi(\cdot, \ee)$.  We also introduce the $\xi$-Szlenk power type and investigate ideal and factorization properties of classes associated with the $\xi$-Szlenk power type.

\end{abstract}



\maketitle

\section{Introduction}

Since its introduction, the Szlenk index has been closely tied to the problem of asymptotically uniformly smooth renormings of Banach spaces.  Knaust, Odell, and Schlumprecht \cite{KOS} showed that a separable Banach space has an equivalent asymptotically uniformly smooth norm, and even with power type modulus, if and only if the Szlenk index of the space does not exceed the first infinite ordinal, $\omega$.  Raja \cite{Raja} extended this result to non-separable spaces, while Godefroy, Kalton, and Lancien \cite{GKL} gave the optimum renorming result regarding the best power type of an equivalent norm of a Banach space in terms of the Szlenk power type, which we define below.  In \cite{CD}, for every ordinal $\xi$, the definitions of $\xi$-asymptotically uniformly smooth and $w^*$-asymptotically uniformly convex operators were given, as well as moduli associated with these properties.  We define these moduli and properties in the next section. These properties generalize the notions of asymptotically uniformly smooth and $w^*$-asymptotically uniformly convex Banach spaces.  We collect some of the results established in \cite{CD} regarding these notions.

\begin{theorem} Let $A:X\to Y$ be an operator and $\xi$ an ordinal.  \begin{enumerate}[(i)]\item $Y$ admits an equivalent norm making $A$ $\xi$-asymptotically uniformly smooth if and only if $Sz(A)\leqslant \omega^{\xi+1}$.   \item $A$ is $\xi$-asymptotically uniformly smooth if and only if $A^*:Y^*\to X^*$ is $w^*$-asymptotically uniformly convex. \item Given $1<p<\infty$, the modulus of $\xi$-asymptotic uniform smoothness of $A$ is of power type $p$ if and only if the modulus of $w^*$-asymptotic uniform convexity of $A^*$ is of power type $q$, where $1/p+1/q=1$. \end{enumerate}

\label{CDtheorem}
\end{theorem}

 It was also discussed in \cite{CD} that for any ordinal $\xi$, there exists an operator $A:X\to Y$ which is $\xi$-asymptotically uniformly smooth such that $Y$ does not admit any equivalent norm such that the modulus of $\xi$-asymptotic uniform smoothness of $A$ has non-trivial power type. Moreover, this operator can be taken to be an identity operator if and only if $\xi>0$.  Thus the question arises of when a given operator with Szlenk index $\omega^{\xi+1}$ can be renormed to be $\xi$-asymptotically uniformly smooth with power type modulus, and in this case what is the best possible exponent. By Theorem \ref{CDtheorem}, answering this question regarding optimal power type of $\xi$-asymptotically uniformly smooth norms for $A$ also solves the problem of $w^*$-$\xi$-asymptotically uniformly convex norms for $A^*$.  

If $A:X\to Y$ is an operator with Szlenk index $\omega^{\xi+1}$, then for every $\ee>0$, there exists a minimum natural number $n$ such that $s_\ee^{\omega^\xi n}(A^*B_{Y^*})$ is empty (we define this notation in the next section). We denote this natural number by $Sz_\xi(A, \ee)$.  We then define $$\textbf{p}_\xi(A)=\underset{\ee\to 0^+}{\lim\sup} \frac{\log  Sz_\xi(A, \ee)}{|\log(\ee)|},$$   noting that this value need not be finite. If $X$ is a Banach space, we write $\textbf{p}_\xi(X)$ in place of $\textbf{p}_\xi(I_X)$. The $\xi=0$ case has been previously studied for spaces, but not for operators, while this value for $\xi>0$ is introduced here for both spaces and operators.    The main results of this work are the following. In the following theorem, $p'$ denotes the conjugate exponent to $p$. 

\begin{theorem*} Let $\xi$ be an ordinal and let $A:X\to Y$ be an operator with $Sz(A)=\omega^{\xi+1}$.   \begin{enumerate}[(i)]\item  $Y$ admits an equivalent norm $|\cdot|$ such that $A:X\to (Y, |\cdot|)$ is $\xi$-asymptotically uniformly smooth norm with power type if and only if $\textbf{\emph{p}}_\xi(A)<\infty$. \item  If $\textbf{\emph{p}}_\xi(A)<\infty$, then $1\leqslant \textbf{\emph{p}}_\xi(A)$, $\textbf{\emph{p}}_\xi(A)$ is the infimum of those $p$ such that $Y$ admits an equivalent norm making $A^*$ $w^*$-$\xi$-asymptotically uniformly convex with power type $p$, and  $\textbf{\emph{p}}_\xi(A)'$ is the supremum of those $p$ such that $Y$ admits an equivalent norm making $A$ $\xi$-asymptotially uniformly smooth with power type $p$. \end{enumerate}

\label{main theorem}
\end{theorem*}

As was shown in \cite{CD}, if $Sz(A)\leqslant \omega^\xi$, $\rho_\xi(\sigma;A:X\to Y)=0$ for all $\sigma>0$, so that $A$ is already $\xi$-AUS with all power types in this case.  Also, if $Sz(A)>\omega^{\xi+1}$, there is no norm $|\cdot|$ on $Y$ making $A:X\to (Y, |\cdot|)$ $\xi$-AUS.  Therefore regarding equivalent norms on $Y$ making $A:X\to Y$ $\xi$-AUS, the Main Theorem  treats the only non-trivial case, $Sz(A)=\omega^{\xi+1}$.  Here we are using the fact shown by Brooker in \cite{BrookerAsplund} that $Sz(A)=\infty$ or there exists an ordinal $\xi$ such that $Sz(A)=\omega^\xi$.

Applying these results to the identity of a Banach space, we obtain the following.

\begin{corollary*}Let $\xi$ be an ordinal and let $X$ be a Banach space with $Sz(X)=\omega^{\xi+1}$.  \begin{enumerate}[(i)]\item  $X$ admits an equivalent $\xi$-asymptotically uniformly smooth norm with power type if and only if $\textbf{\emph{p}}_\xi(X)<\infty$. \item  If $\textbf{\emph{p}}_\xi(X)<\infty$, then $1\leqslant \textbf{\emph{p}}_\xi(X)$, $\textbf{\emph{p}}_\xi(X)$ is the infimum of those $p$ such that $X$ admits a norm making $X^*$ $w^*$-$\xi$-asymptotically uniformly convex with power type $p$,  and $\textbf{\emph{p}}_\xi(X)'$ is the supremum of those $p$ such that $X$ admits an equivalent $\xi$-asymptotically uniformly smooth norm with power type $p$.\end{enumerate}
\label{main corollary}
\end{corollary*}

In the case that $\xi=0$, this corollary recovers the precise renorming result from \cite{GKL}.

In \cite{DK}, it was shown that for an infinite dimensional Banach space $X$ with separable dual, the conjugate $\textbf{p}_0(X)'$ of the Szlenk power type $\textbf{p}_0(X)$ of $X$ is equal to the supremum of those $p>1$ such that $X$ satisfies subsequential $\ell_p$-upper tree estimates.  In this paper, we define for every ordinal $\xi$ and every $1<p<\infty$ what it means for an operator to satisfy $\xi$-$\ell_p$ \emph{upper tree estimates}.  In analogy with the $\xi=0$ case, we prove the following.  

\begin{theorem} Let $\xi$ be an ordinal and let $A:X\to Y$ be an operator such that $1\leqslant \textbf{\emph{p}}(A)<\infty$.  Then $\textbf{\emph{p}}_\xi(A)'$ is the supremum of those $p$ such that $A$ satisfies $\xi$-$\ell_p$ upper tree estimates. 

\end{theorem}

A  part of the work here follows closely the arguments from \cite{GKL}.  However, the higher ordinal arguments require complex blockings of the branches of weakly null trees not required in the $\xi=0$ case.   For technical reasons, we will rely heavily on special convex combinations and repeated averages introduced in \cite{C}, based on the repeated averages hierarchy introduced in \cite{AMT}.   We will require several combinatorial stabilization results for functions defined on the branches of special trees. Namely, a function defined on a set with a ``leveled'' tree structure having large averages on every branch with respect to special probability measures has a ``full'' subset which is uniformly large on each level.   All of the notions referred to in the following result will be defined in Section $3$.   Here, for every ordinal $\xi$ and every $n\in \nn$, we will define a set $\Gamma_{\xi, n}$ with order $\omega^\xi n$ consisting of $n$ ``levels,'' denoted $\Lambda_{\xi,n,1}, \ldots, \Lambda_{\xi, n,n}$.   The set $\Gamma_{\xi,n}$ is partially ordered by a relation $\preceq$, $MAX(\Gamma_{\xi, n})$ denotes the $\preceq$-maximal members of $\Gamma_{\xi, n}$,  and a function $\mathbb{P}_{\xi, n}:\Gamma_{\xi, n}\to [0,1]$ is defined such that for every  $t\in MAX(\Gamma_{\xi, n})$ and for each $1\leqslant i\leqslant n$, $1=\sum_{\Lambda_{\xi, n, i}\ni s\preceq t}\mathbb{P}_{\xi, n}(s)$. A function $\theta:\Gamma_{\xi, n}\to \Gamma_{\xi,n}$ is called \emph{monotone} if $\theta(s)\prec \theta(t)$ when $s\prec t$. If $\theta:\Gamma_{\xi, n}\to \Gamma_{\xi,n}$ is monotone, an \emph{extension} of $\theta$ is a function $e:MAX(\Gamma_{\xi, n})\to MAX(\Gamma_{\xi, n})$ such that for every $t\in MAX(\Gamma_{\xi, n})$, $\theta(t)\preceq e(t)$.  The following is a simplified version of our main combinatorial result.  

\begin{theorem} For any ordinal $\xi$, any natural number $n$, any $\delta>0$, and any bounded function $f:\Gamma_{\xi, n}\to \rr$, there exist numbers $a_1, \ldots, a_n\in \rr$, a monotone function $\theta:\Gamma_{\xi, n}\to \Gamma_{\xi, n}$, and an extension $e$ of $\theta$ such that \begin{enumerate}[(i)] \item for every $t\in MAX(\Gamma_{\xi, n})$, $\sum_{s\preceq e(t)}\mathbb{P}_{\xi, n}(s)f(s)\leqslant \delta+\sum_{i=1}^n a_i$, and \item for every $1\leqslant i\leqslant n$ and every $s\in \Lambda_{\xi, n, i}$, $f(\theta(s))\geqslant a_i$. \end{enumerate}

\end{theorem}

The particular case $n=1$ shows that in order to find a subtree of $\Gamma_\xi$ of full order $\omega^\xi$ which has uniformly large values, it is sufficient to check that one has uniformly large averages on all maximal linearly ordered subsets of $\Gamma_\xi$ (called \emph{branches} of $\Gamma_\xi$) with respect to a specific choice of weights.  One phenomenon which was used several times implicitly in \cite{CD} will be vital to this work.  This phenomenon involves beginning with multiple functions which cannot be made to be uniformly large on a subtree of $\Gamma_\xi$ with full order and finding a single branch such that all of the functions have small averages down this branch.  To that end, for $\ee\in \rr$, let us say a function $f:\Gamma_\xi\to \rr$ is $\ee$-\emph{large} if for any $\delta>0$, there exists a monotone function $\theta:\Gamma_\xi\to \Gamma_\xi$ such that for all $t\in \Gamma_\xi$, $f\circ \theta(t)\geqslant \ee-\delta$.  

\begin{theorem} Let $\xi$ be an ordinal. Fix $\ee_1, \ldots, \ee_n\in \rr$ and bounded functions $f_1, \ldots, f_n:\Gamma_\xi \to \rr$ such that for each $1\leqslant i\leqslant n$, $f_i$ is not $\ee_i$-large.  Then $$\underset{t\in MAX(\Gamma_\xi)}{\inf} \underset{1\leqslant i\leqslant n}{\max}\sum_{s\preceq t} \mathbb{P}_\xi(s)[f_i(s)-\ee_i]\leqslant 0.$$

\end{theorem}

\section{Definitions and background}

For a Banach space $X$, a $w^*$-compact subset $K$ of $X^*$, and $\ee>0$, we let $s_\ee(K)$ denote the subset of $K$ consisting of those $x^*\in K$ such that for every $w^*$-neighborhood $V$ of $x^*$, $\text{diam}(V\cap K)>\ee$.   It will be convenient for our purposes to define $s_\ee(K)=K$ for every $\ee\leqslant 0$.  We define the higher order derivatives $s_\ee^\xi(K)$ by transfinite induction by $s_\ee^0(K)=K$, $s_\ee^{\xi+1}(K)=s_\ee(s_\ee^\xi(K))$, and $s_\ee^\xi(K)=\cap_{\zeta<\xi}s^\zeta_\ee(K)$ when $\xi$ is a limit ordinal.  For $\ee>0$, we let $Sz(K, \ee)$ denote the smallest ordinal $\xi$ such that $s^\xi_\ee(K)=\varnothing$, if such a $\xi$ exists.  If $s_\ee^\xi(K)\neq \varnothing$ for all $\xi$, we write $Sz(K, \ee)=\infty$.   We agree to the convention that $\xi<\infty$ for every ordinal $\xi$.  We let $Sz(K)=\sup_{\ee>0}Sz(K, \ee)$.  For a Banach space $X$, we let $Sz(X, \ee)=Sz(B_{X^*}, \ee)$, $Sz(X)=Sz(B_{X^*})$.    If $A:X\to Y$ is an operator, we let $Sz(A, \ee)=Sz(A^*B_{Y^*}, \ee)$ and $Sz(A)=Sz(A^*B_{Y^*})$.  We recall that $A:X\to Y$ is an Asplund operator if and only if $Sz(A)<\infty$, and $X$ is an Asplund space if and only if $Sz(X)<\infty$ (see, for example, \cite{BrookerAsplund} and the references therein). 

We define some other derivations which we will need.  For an ordinal $\xi>0$, $\ee\in \rr$, we define $s_{\xi, \ee}^0(K)=K$, $s_{\xi, \ee}^{\beta+1}(K)=s_\ee^{\omega^\xi}(s_{\xi, \ee}^\beta(K))$, and if $\beta$ is a limit ordinal, $s_{\xi, \ee}^\beta=\cap_{\zeta<\beta}s_{\xi, \ee}^\zeta(K)$.  It follows from standard induction proofs that for every ordinal $\beta$, $s_{\xi, \ee}^\beta(K)=s_\ee^{\omega^\xi \beta}(K)$. We let $Sz_\xi(K, \ee)$ be the smallest $\beta$ such that $s_{\xi, \ee}^\beta(K)=\varnothing$ if such a $\beta$ exists and let $Sz_\xi(K)=\infty$ otherwise.    The symbols $Sz_\xi(K)$, $Sz_\xi(X, \ee)$, $Sz_\xi(A, \ee)$ are defined analogously to the previous paragraph.  If $A:X\to Y$ is  an operator with  $Sz(A)>\omega^\xi$, then $\epsilon_\xi(A):=\sup\{\ee>0: Sz_{\omega^\xi}(A, \ee)>1\}$ is a positive, finite value.

Note that if $K\subset X^*$ is weak$^*$-compact,  $Sz(K)\leqslant \omega^{\xi+1}$ if and only if  $Sz_\xi(K,\ee)$ takes values in the natural numbers for all $\ee>0$.  Therefore if $Sz(K)\leqslant \omega^{\xi+1}$,  we  may define the $\xi$-\emph{Szlenk power type of} $K$ by $$\textbf{p}_\xi(K)=\underset{\ee\to 0^+}{\lim\sup} \frac{\log Sz_\xi(K, \ee)}{|\log(\ee)|}.$$ For convenience, if $Sz(K)>\omega^{\xi+1}$, we define $\textbf{p}_\xi(K)=\infty$. If $A$ is an operator, we let $\textbf{p}_\xi(A)=\textbf{p}_\xi(A^*B_{Y^*})$. If $X$ is a Banach space, we write $\textbf{p}_\xi(X)$ in place of $\textbf{p}_\xi(I_X)$.     In general, $\textbf{p}_\xi(K)$ need not be finite, even if $Sz(K)=\omega^{\xi+1}$.  However, $\textbf{p}_0(X)$ is automatically finite when $X$ is a Banach space with $Sz(X)\leqslant \omega$,  which we discuss in the final section.  We note that if $Sz(K)\leqslant \omega^\xi$, then $Sz_\xi(K,\ee)=1$ for every $\ee>0$, whence $\textbf{p}_\xi(K)=0$.  We will see in Lemma \ref{lemma1} that the converse is also true when $K$ is convex, and in particular if $A$ is an operator with $Sz(A)=\omega^{\xi+1}$, $\textbf{p}_\xi(A)\geqslant 1$.  Since the Szlenk index of an operator cannot take values strictly between $\omega^\xi$ and $\omega^{\xi+1}$, which was shown in \cite{BrookerAsplund}, we will see that $\textbf{p}_\xi(A)=0$ if and only if $Sz(A)\leqslant \omega^\xi$.

We note that $\textbf{p}_\xi(A)$ is invariant under renorming $Y$. Indeed, suppose that $C>0$ and $|\cdot|$ is a norm on $Y$ such that $C^{-1} B_Y \subset B_Y^{|\cdot|}\subset CB_Y$. Then for any $\ee>0$ and any ordinal $\zeta$, $$s^\zeta_\ee(A^*B_{Y^*}^{|\cdot|})\subset s^\zeta_\ee(CA^*B_{Y^*})= Cs^\zeta_{\ee/C}(B_{Y^*}).$$  From this it follows that $Sz_\xi(A^*B_{Y^*}^{|\cdot|}, \ee)\leqslant Sz_\xi(A^*B_{Y^*}, \ee/C)$ and \begin{align*} \underset{\ee\to 0^+}{\lim\sup}\frac{\log Sz_\xi(A^*B_{Y^*}^{|\cdot|}, \ee)}{|\log(\ee)|} & \leqslant \underset{\ee\to 0^+}{\lim\sup}\frac{\log Sz_\xi(A^*B_{Y^*}, \ee/C}{|\log(\ee)|} \\ & = \underset{\ee\to 0^+}{\lim\sup}\frac{\log Sz_\xi(A^*B_{Y^*}, \ee/C}{|\log(\ee/C)|}.\end{align*}

Given a set $\Lambda$, we let $\Lambda^{<\nn}$ denote the finite sequences in $\Lambda$, including the empty sequence $\varnothing$.  We order $\Lambda^{<\nn}$ by $\prec$, where $s\prec t$ if and only if $s$ is a proper initial segment of $t$. Given any set $\Lambda$ and a subset $S$ of $\Lambda^{<\nn}$, we let $MAX(S)$ denote the $\prec$-maximal members of $S$ and $\Pi S=\{(s,t)\in S\times S: s\preceq t\in MAX(S)\}$.   Given $t\in \Lambda^{<\nn}$, we let $|t|$ denote the length of $t$. For $0\leqslant i\leqslant |t|$, we let $t|_i$ denote the initial segment of $t$ having length $i$.  If $t\neq \varnothing$, we let $t^-$ denote the $\prec$-maximal proper initial segment of $t$.  We let $s\cat t$ denote the concatenation of $s$ and $t$. We say $T\subset \Lambda^{<\nn}$ is a \emph{tree} if it is downward closed with respect to $\prec$.

Given a set $\Lambda$ and a subset $T$ of $\Lambda^{<\nn}$, we let $T'=T\setminus MAX(T)$ and note that $T'$ is a tree if $T$ is.  We define the higher order derived trees of $T$ by transfinite induction.  We let $T^0=T$,   $T^{\xi+1}=(T^\xi)'$, and if $\xi$ is a limit ordinal, $T^\xi=\cap_{\zeta<\xi}T^\zeta$.  Note that there exists a smallest ordinal $\xi$ such that $T^\xi=T^{\xi+1}$.  If $T^\xi=\varnothing$ for this $\xi$, we say $T$ is \emph{well-founded} and let $o(T)=\xi$.  

We say a subset $B$ of $\Lambda^{<\nn}\setminus\{\varnothing\}$ is a $B$-\emph{tree} if $B\cup \{\varnothing\}$ is a tree.  All of the notions above regarding trees can be relativized to $B$-trees.

Given a Banach space $X$, a topology $\tau$ on $X$, and a tree $T$, we say a collection $(x_t)_{t\in T}\subset X$ is $\tau$-\emph{closed} if for every ordinal $\xi$ and every $t\in T^{\xi+1}$, $$x_t\in \overline{\{x_s: s\in T^\xi, s^-=t\}}^\tau.$$ If $B$ is a $B$-tree, we say a collection $(x_t)_{t\in B}\subset X$ is $\tau$-\emph{null} provided that for every ordinal $\xi$ and every $t\in (B\cup \{\varnothing\})^{\xi+1}$, $$0\in \overline{\{x_s: s\in B^\xi, s^-=t\}}^\tau.$$    Given $\ee>0$, a tree $T$, and an operator $A:X\to Y$, we say a collection $(x_t)_{t\in T}$ is $(A, \ee)$-\emph{large} provided that for every $\varnothing\neq t\in T$, $\|Ax_t\|>\ee$.

Given an operator $A:X\to Y$, $\sigma\geqslant 0$, and an ordinal $\xi$, we let $$\rho_\xi(\sigma;A)=\sup \inf \{\|y+ A x\|-1: t\in B, x\in \text{co}(x_s:s\preceq t)\},$$ where the supremum is taken over all $y\in B_Y$, all $B$-trees $B$ with $o(B)=\omega^\xi$, and all weakly null collections $(x_t)_{t\in B}\subset \sigma B_X$.  Then $\rho_\xi(\cdot;A)$ is the \emph{modulus of} $\xi$-\emph{asymptotic uniform smoothness of} $A$. Note that either $Y=\{0\}$ or $\rho_\xi(\sigma;A)\geqslant 0$.  We say $A$ is $\xi$-\emph{asymptotically uniformly smooth} (in short, $\xi$-AUS)  if $A=0$ or if $\lim_{\sigma\to 0^+}\rho_\xi(\sigma;A)/\sigma=0$. We say $A:X\to Y$ \emph{admits a} $\xi$-\emph{asymptotically uniformly smooth norm}  if there exists an equivalent norm $|\cdot|$ on $Y$ such that $A:X\to (Y, |\cdot|)$ is $\xi$-AUS.  For $1<p<\infty$, we say $A:X\to Y$ is $\xi$-\emph{AUS with power type} $p$ if there exists a constant $C$ such that for all $0<\sigma\leqslant 1$, $\rho_\xi(\sigma;A)\leqslant C\sigma^p$.   Note that since $\rho(\sigma;A)\leqslant \|A\|\sigma$ for all $\sigma>0$, the existence of some constant $C$ such that $\rho_\xi(\sigma;A)\leqslant C\sigma^p$ for all $\sigma\in (0,1]$ implies the existence of some constant $C$ such that $\rho_\xi(\sigma;A)\leqslant C\sigma^p$ for all $0<\sigma$.   We say $A:X\to Y$ is $\xi$-\emph{AUS with power type} $\infty$ if there exists a constant $C$ such that for any $y\in Y$, any $b>0$, any $B$-tree $B$ with $o(B)=\omega^\xi$, and any weakly null collection $(x_t)_{t\in B}\subset bB_X$, $$\inf\{\|y+Ax\|: t\in B, x\in \text{co}(x_s:s\preceq t)\} \leqslant \max\{\|y\|, Cb\}.$$ 

 We let $$\delta^{w^*}_\xi(\sigma;A^*)=\inf \sup\{\|y^*+\sum_{s\preceq t} z^*_s\|-1: t\in B\},$$ where the infimum is taken over all $y^*$ with $\|y^*\|= 1$, all $B$-trees $B$ with $o(B)=\omega^\xi$,  and all $w^*$-null, $(A^*, \sigma)$-large collections $(z^*_s)_{s\in B}$. We remark that this modulus was defined in \cite{CD} to be the infimum over all $y^*$ with $\|y^*\|\geqslant 1$ rather than all $y^*$ with $\|y^*\|=1$, but each of the arguments from that paper go through unchanged with this definition.   The function $\delta_\xi^{w^*}$ is called the \emph{modulus of} $w^*$-$\xi$-\emph{asymptotic uniform convexity of} $A^*$. We say $A^*$ is $w^*$-$\xi$-\emph{asymptotically uniformly convex} (in short, $w^*$-$\xi$-AUC) provided $\delta^{w^*}_\xi(\sigma;A^*)>0$ for all $\sigma>0$.  We say $A^*:Y^*\to X^*$ \emph{admits a} $w^*$-$\xi$-\emph{asymptotically uniformly convex norm} if there exists an equivalent norm $|\cdot|$ on $Y$ such that $A^*:(Y^*, |\cdot|)\to X^*$ is $w^*$-$\xi$-AUC. Here, $|\cdot|$ denotes both the norm on $Y$ and its dual norm on $Y^*$.   We say $A^*$ is $w^*$-$\xi$-AUC with power type $p$ if there exists $c>0$ such that for all $0<\tau\leqslant 1$, $\delta^{w^*}_\xi(\tau;A^*)\geqslant c\tau^p$. We note that if $A$ is an identity operator and if $\xi=0$, these notions recover the usual notions of asymptotic uniform smoothness and $w^*$-asymptotic uniform convexity. 

It was shown in \cite{CD} that if $A:X\to Y$ is $\xi$-AUS, then for every $\ee>0$, there exists some $\delta\in (0,1)$ such that $s^{\omega^\xi}_\ee(A^*B_{Y^*})\subset (1-\delta)A^*B_{Y^*}$.  A standard homogeneity argument yields that $Sz(A)\leqslant \omega^{\xi+1}$.  From this it follows that the only operators which may admit a $\xi$-AUS norm are those operators with $Sz(A)\leqslant\omega^{\xi+1}$.  We will implicitly use this fact throughout.   

Given an operator $A:X\to Y$, $1<p<\infty$, and an ordinal $\xi$, we let $\textbf{t}_{\xi, p}(A)$ denote the infimum of all $C> 0$ such that for every $y\in Y$, every $\sigma>0$, every $B$-tree $B$ with $o(B)=\omega^\xi$, and every weakly null collection $(x_t)_{t\in B}\subset \sigma B_X$, $$\inf\{\|y+Ax\|^p: t\in B, x\in \text{co}(x_s:s\preceq t)\}\leqslant \|y\|^p+C\sigma^p.$$  We define $\textbf{t}_{\xi, \infty}(A)$ to be the infimum of all $C>0$ such that for every $y\in Y$, every $\sigma>0$, every $B$-tree $B$ with $o(B)=\omega^\xi$, and every weakly null collection $(x_t)_{t\in B}\subset \sigma B_X$, $$\inf\{\|y+Ax\|: t\in B, x\in \text{co}(x_s:s\preceq t)\}\leqslant \max\{\|y\|, C\sigma\}.$$ It follows from the definition that $A$ is $\xi$-AUS with power type $\infty$ if and only if $\textbf{t}_{\xi, \infty}(A)<\infty$.  For future use, we isolate the following elementary proposition, which shows that the same characterization holds for $1<p<\infty$.

\begin{proposition} Let $\xi$ be an ordinal, $A:X\to Y$ an operator, and $1<p<\infty$.  Then $A$ is $\xi$-AUS with power type $p$ if and only if $\textbf{\emph{t}}_{\xi, p}(A)<\infty$.  In this case, each of the quantities $\textbf{\emph{t}}_{\xi, p}(A)$, $\sup_{\sigma\in (0,1)}\rho_\xi(\sigma;A)/\sigma^p$ can be majorized by a quantity depending on the other, $p$, and $\|A\|$.   

\label{mymaxinfo}
\end{proposition}   

\begin{proof} 

Assume $A$ is $\xi$-AUS with power type $p$ and let $C'=\sup_{\sigma\in (0,1)} \rho_\xi(\sigma;A)/\sigma^p<\infty$. If $C'\leqslant 0$, $$\inf\{\|y+Ax\|^p: t\in B, x\in \text{co}(x_s:s\preceq t)\}\leqslant 1,$$ and $\textbf{t}_{\xi, p}(A)=0$. Now suppose $C'>0$.    We will choose $C\geqslant \|A\|^p$ such that for any $y\in S_Y$, any $\sigma>0$, any $B$-tree $B$ with $o(B)=\omega^\xi$, and any weakly null collection $(x_t)_{t\in B}\subset \sigma B_X$, $$\inf\{\|y+Ax\|^p: t\in B, x\in \text{co}(x_s:s\preceq t)\}\leqslant 1+C\sigma^p.$$  By homogeneity, this will handle the case of all $y\neq 0$, and the case $y=0$ is handled by the inequality $C\geqslant \|A\|^p$.   First suppose that $0\leqslant C'\sigma^p\leqslant 1/2$. Then, as one can see by considering the power series of $(1+x)^p$ for $x\in [0,1/2]$ and using the definition of $\rho_\xi(\sigma;A)$, there exists a constant $b$ (depending on $p$) such that  $$\inf\{\|y+Ax\|^p: t\in B, x\in \text{co}(x_s:s\preceq t)\}\leqslant (1+C'\sigma^p)^p\leqslant 1+bC'\sigma^p.$$  If $C'\sigma^p>1/2$, $(2C')^{1/p}\sigma>1$, and by the triangle inequality, \begin{align*} \inf\{\|y+Ax\|^p: t\in B, x\in \text{co}(x_s:s\preceq t)\} & \leqslant (1+\|A\|\sigma)^p <((2C')^{1/p}+\|A\|)^p\sigma^p  \\ & <1+((2C')^{1/p}+\|A\|)^p\sigma^p.\end{align*}  Thus we may take $C=\max\{bC', ((2C')^{1/p}+\|A\|)^p\}$.

Next, suppose $C=\textbf{t}_{\xi, p}(A)<\infty$.   Fix $y\in B_Y$, $\sigma>0$, a $B$-tree $B$ with $o(B)=\omega^\xi$, and a weakly null collection $(x_t)_{t\in B}\subset \sigma B_X$.   Then if $0\leqslant C\sigma^p\leqslant 1/2$, by again considering the power series expansion of $(1+x)^{1/p}-1$,  $$\inf\{\|y+Ax\|-1: t\in B, x\in \text{co}(x_s:s\preceq t)\} \leqslant (\|y\|^p+C\sigma^p)^{1/p}-1\leqslant 2C\sigma^p.$$  If $1/2<C\sigma^p$, \begin{align*}\inf\{\|y+Ax\|-1: t\in B, x\in \text{co}(x_s:s\preceq t)\} & \leqslant (1+C\sigma^p)^{1/p}-1<(3C\sigma^p)^{1/p} \\ & = 3^{1/p}C(C\sigma^p)^{\frac{1-p}{p}}\sigma^p<2^{\frac{p-1}{p}}3^{1/p}C\sigma^p. \end{align*}

\end{proof}

The following was the main theorem of \cite{CIllinois}, and is vital to this work. 

\begin{theorem} Fix an ordinal $\xi$.  \begin{enumerate}[(i)]\item For any operator $A:X\to Y$, $Sz(A)\leqslant \omega^\xi$ if and only if for any $B$-tree $B$ with $o(B)=\omega^\xi$, any bounded, weakly null collection $(x_t)_{t\in B}$, and any $\ee>0$, there exists $t\in B$ and $x\in \text{\emph{co}}(x_s:s\preceq t)$ such that $\|Ax\|<\ee$.  \item There exists a constant $c>1$ such that if $A:X\to Y$ is an operator and $\ee>0$ is such that $Sz(A, c\ee)>\xi$, there exists a $B$-tree $B$ with $o(B)=\xi$ and a weakly null collection $(x_t)_{t\in B}\subset B_X$ such that for every $t\in B$ and every $x\in \text{\emph{co}}(x_s:s\preceq t)$, $\|Ax\|\geqslant \ee$.   \end{enumerate}

\label{illinoistheorem}
\end{theorem}

The preceding lemma was not presented using this language, but rather in terms of the \emph{weakly null index}.  We explain the equivalence of the statement here to the main theorem of \cite{CIllinois}.     For a collection $\hhh$ of finite sequences in $B_X$, we let $(\hhh)_w'$ denote the collection of all sequences $t\in \hhh$ such that for every weak neighborhood $U$ of $0$ in $X$, there exists $x\in U\cap B_X$ such that $t\cat x\in \hhh$.  We define the higher order derivations $(\hhh)_w^\xi$ and the associated order $o_w(\hhh)$ as usual.  Given an operator $A:X\to Y$ and $\ee>0$,  we let $\hhh^{A^*B_{Y^*}}_\ee$ denote the collection consisting of the empty sequence together with all $(x_i)_{i=1}^n\in B_X^{<\nn}$ such that for every $x\in \text{co}(x_i:1\leqslant i\leqslant n)$, $\|Ax\|\geqslant \ee$.  Then for any $\ee>0$ and any ordinal $\xi$, $o_w(\hhh^{A^*B_{Y^*}}_\ee)>\xi$ if and only if there exists a $B$-tree $B$ with $o(B)=\xi$ and a weakly null collection $(x_t)_{t\in B}\subset B_X$ such that for every $t\in B$, $(x_{t|_i})_{i=1}^{|t|}\in \hhh^{A^*B_{Y^*}}_\ee$ (that is, for every $t\in B$ and every $x\in \text{co}(x_s:s\preceq t)$, $\|Ax\|\geqslant \ee$).  Indeed, it was shown in \cite{CIllinois} that if $o_w(\hhh^{A^*B_{Y^*}}_\ee)>\xi$, then such a $B$ and $(x_t)_{t\in B}\subset B_X$ exist.  For the converse, a straightforward proof by induction shows that if such a $B$ and weakly null $(x_t)_{t\in B}\subset B_X$ exist, then for every $\zeta<\xi$ and $t\in B^\zeta$, $(x_{t|_i})_{i=1}^{|t|}\in (\hhh^{A^*B_{Y^*}}_\ee)^\zeta_w$.  In fact, the definition of weakly null collection is given precisely to facilitate this inductive proof.   It follows that $\varnothing\in (\hhh_\ee^{A^*B_{Y^*}})_w^\xi$ and $o_w(\hhh^{A^*B_{Y^*}}_\ee)>\xi$.  This yields the indicated equivalence.   In \cite{CIllinois}, it was shown that there exists a constant $c>1$ such that for every $0<\ee<\ee_1$, $$o_w(\hhh^{A^*B_{Y^*}}_{\ee_1})\leqslant Sz(A, \ee), \hspace{5mm} Sz(A, c\ee)\leqslant o_w(\hhh^{A^*B_{Y^*}}_\ee).$$  The second inequality and the preceding discussion yield Theorem \ref{illinoistheorem}$(ii)$.

\begin{corollary} Fix an ordinal $\xi$. \begin{enumerate}[(i)]\item For any operator $A:X\to Y$, $Sz(A)\leqslant \omega^\xi$ if and only if $\textbf{\emph{t}}_{\xi, \infty}(A)=0$.  \item For any $1<p<\infty$ and any operator $A:X\to Y$ with $Sz(A)\leqslant \omega^{\xi+1}$, $\textbf{\emph{p}}_\xi(A)\leqslant p$ if and only if for every $p<q<\infty$, there exists a constant $C$ such that if $0<\ee<1$ and $k\in \nn$ are such that there exists a $B$-tree $B$ with $o(B)=\omega^\xi k$ and a weakly null collection $(x_t)_{t\in B}\subset B_X$ such that for every $t\in B$ and every $x\in \text{\emph{co}}(x_s:s\preceq t)$, $\|Ax\|\geqslant \ee$, then $\ee\leqslant C/k^{1/q}$.  \end{enumerate}

\label{illinoiscorollary}
\end{corollary}

\begin{proof}
$(i)$ If $\mathcal{B}$ denotes the collection of all weakly null collections $(x_t)_{t\in B}\subset B_X$, where $B$ is a $B$-tree with $o(B)=\omega^\xi$, it is quite clear that $$\sup_{y\in Y,(x_t)_{t\in B}\in \mathcal{B}} \inf \{\|y+Ax\|-\|y\|: t\in B, x\in \text{co}(x_s:s\preceq t)\}\leqslant 0$$ if and only if $$\sup_{(x_t)_{t\in B}\in \mathcal{B}}\inf\{\|Ax\|: t\in B, x\in \text{co}(x_s:s\preceq t)\}\leqslant 0.$$   By homogeneity, the former is equivalent to $\textbf{t}_{\xi, \infty}(A)=0$, while the latter is equivalent to $Sz(A)\leqslant \omega^\xi$ by Theorem \ref{illinoistheorem}.

$(ii)$ Suppose $\textbf{p}_\xi(A)\leqslant p$ and fix $p<q<\infty$.   Then $C=\sup_{0<\ee<1}\ee^q Sz_\xi(A, \ee)<\infty$.   Suppose $o(B)=\omega^\xi k$ and $(x_t)_{t\in B}\subset B_X$ are such that $\|Ax\|\geqslant \ee$ for all $t\in B$ and $x\in \text{co}(x_s:s\preceq t)$. Then $Sz_\xi(A, \ee/2)> k$ by Theorem \ref{illinoistheorem} and the discussion preceding the corollary.  From this it follows that $(\ee/2)^q\leqslant C/k$, and  $\ee\leqslant 2C^{1/q}/k^{1/q}$.   Next, let $c>1$ be the constant from Theorem \ref{illinoistheorem}$(ii)$. Suppose that for any $p<q<\infty$, there exists a constant $C$ as in item $(ii)$ of the corollary and fix such $q$ and $C$.   Suppose that $Sz_\xi(A,\ee)=k+1$ for some $\ee\in (0,1)$. If $k=0$, then $\log Sz_\xi(A, \ee)/|\log(\ee)|=0$.  Otherwise by Theorem \ref{illinoistheorem}, there exists a $B$-tree $B$ with $o(B)=\omega^\xi k$ such that for all $t\in B$ and $x\in\text{co}(x_s:s\preceq t)$, $\|Ax\|\geqslant \ee/c$.   Then $\ee\leqslant cC/k^{1/q}$, and $Sz_\xi(A, \ee)=k+1\leqslant 1+\ee^{-q} c^qC^q<(1+c^qC^q)\ee^{-q}$.  From this it follows that $\textbf{p}_\xi(A)\leqslant q$.  Since this holds for all $p<q<\infty$, $\textbf{p}_\xi(A)\leqslant p$.    

\end{proof}

\section{Combinatorics}

We recall some important $B$-trees from \cite{C}.  Given a sequence $(\zeta_i)_{i=1}^n$ of ordinals and an ordinal $\zeta$, we let $\zeta+(\zeta_i)_{i=1}^n=(\zeta+\zeta_i)_{i=1}^n$.  Given a set $S$ of sequences of ordinals and an ordinal $\zeta$, we let $\zeta+S=\{\zeta+s: s\in S\}$.  We will define the $B$-tree $\Gamma_\xi$ on $[0, \omega^\xi)$ by induction on $\xi$. We let $\Gamma_0=\{(0)\}$.  If $\Gamma_\xi$ has been defined, we let $\Gamma_{\xi, 1}=\Gamma_\xi$ and let $$\Gamma_{\xi, n+1}=(\omega^\xi n+\Gamma_\xi)\cup \{t\cat u:  t\in MAX(\omega^\xi n+\Gamma_\xi), u\in \Gamma_{\xi, n}\Bigr\}.$$ We let $\Gamma_{\xi+1}=\cup_{n=1}^\infty \Gamma_{\xi, n}$ and note that this is a totally incomparable union. For example, for each $n\in \nn$, $$\Gamma_{0, n}= \{(n-1, \ldots, n-k): 1\leqslant k\leqslant n\}$$ and $$\Gamma_1= \bigcup_{n\in \nn} \Gamma_{0,n}= \{(n-1, \ldots, n-k): n\in \nn, 1\leqslant k\leqslant n\}.$$   Last, if $\xi$ is a limit ordinal, we let $$\Gamma_\xi=\cup_{\zeta<\xi}(\omega^\zeta+\Gamma_{\zeta+1}),$$ and note that this is also a totally incomparable union. For any ordinal $\xi$ and any natural number $n$, $o(\Gamma_{\xi, n})=\omega^\xi n$.

Note that we may canonically identify sequences of pairs with pairs of sequences having the same length.  That is, we identify the sequence $((a_i, b_i))_{i=1}^n$ with the pair of sequences $((a_i)_{i=1}^n, (b_i)_{i=1}^n)$.  We use this identification freely throughout.  Given a directed set $D$, a set $\Lambda$,  and a subset $S$ of $\Lambda^{<\nn}$, we let $$SD=\{(\zeta_i, U_i)_{i=1}^n: n\in \nn, (\zeta_i)_{i=1}^n\in S, U_i\in D\}.$$   If $S$ is a $B$-tree, so is $SD$.  Moreover, for any ordinal, it is easy to see that $(SD)^\xi=S^\xi D$, thus $SD$ is well-founded if $S$ is.   Given $B$-trees $S,T$, a function $\theta:S\to T$ is called \emph{monotone} provided that for any $s,s'\in S$ with $s\prec s'$, $\theta(s)\prec \theta(s')$.  Given sets $\Lambda_1, \Lambda_2$ and  subsets $S_1, S_2$ of $\Lambda_1^{<\nn}$, $\Lambda_2^{<\nn}$, we say a function $\theta:S_1D\to S_2D$ is a \emph{pruning} provided that it is monotone and if $s=s_1\cat (\zeta, U)\in S_1D$ and $\theta(s)=s_2\cat (\zeta', U')$, $U\leqslant_D U'$.    If $\theta:S_1D\to S_2D$ is a pruning and $e:MAX(S_1D)\to MAX(S_2D)$ is a function such that for any $t\in MAX(S_1D)$, $\theta(t)\preceq e(t)$, we say the pair $(\theta, e)$ is an \emph{extended pruning}.  We write $(\theta,e):S_1D\to S_2D$.  We remark that if $S,T$ are well-founded $B$-trees, then there exists a monotone function $\theta:S\to T$ if and only if there exists a length-preserving monotone function $\theta:S\to T$ if and only if $o(S)\leqslant o(T)$.  Thus if $o(S)\leqslant o(T)$ for well-founded $B$-trees,  then for any directed set $D$, we may define an extended pruning from $SD$ to $TD$ by first fixing a length-preserving monotone function $\theta:S\to T$ and then define $\phi:SD\to TD$ by $\phi((\zeta_i, U_i)_{i=1}^n)=(\mu_i, U_i)_{i=1}^n$, where $(\mu_i)_{i=1}^n=\phi((\zeta_i)_{i=1}^n)$.  Since $T$ and $TD$ are well-founded, for any $s\in MAX(SD)$, $\phi(s)$ has some maximal extension, say $e(s)$.   Then $(\phi,e):SD\to TD$ is an extended pruning.  The purpose of the sets  $TD$ is to easily index weakly null collections.  Let $X$ be a Banach space and let $D$ be a weak neighborhood basis at $0$ in $X$.  Let us say that a collection of vectors $(x_t)_{t\in TD}\subset X$ is \emph{normally weakly null} provided that for any $t=t_1\cat (\zeta, U)\in TD$, $x_t\in U$.  It is straightforward to check that $(x_t)_{t\in TD}$ is weakly null if it is normally weakly null.  Furthermore, if $\theta:ST\to TD$ is a pruning and if $(x_t)_{t\in TD}$ is normally weakly null, then the collection $(u_t)_{t\in SD}:=(x_{\theta(t)})_{t\in SD}$ is also normally weakly null. 

For convenience, once a directed set $D$ is fixed, we will write $\Omega_\xi$ in place of $\Gamma_\xi D$ and $\Omega_{\xi, n}$ in place of $\Gamma_{\xi, n}D$. For the remainder of this section, we will assume that $D$ is some fixed directed set and $\Omega_{\xi, n}=\Gamma_{\xi, n}D$.  We note that every member $t$ of $\Omega_{\xi+1}$ admits a unique representation of the form $$t=\bigl((\omega^\xi(n-1)+t_1)\cat \ldots \cat (\omega^\xi(n- m)+t_{n-m}), \sigma\bigr),$$ where $1\leqslant m\leqslant n$,  $t_i\in \Gamma_\xi$, and for each $1\leqslant i<n-m$, $t_i\in MAX(\Gamma_\xi)$, and $\sigma$ is a sequence in $D$.   In this case, we say $t$ is on \emph{level} $m$.   We let $\Lambda_{\xi, n, m}$ denote the members of $\Omega_{\xi, n}$ which are on level $m$. For example, $$(\omega\cdot 2+2, \omega\cdot 2+1, \omega\cdot 2)=\omega\cdot 2+(2,1,0)$$ is a maximal member of $\Lambda_{1, 3,1}$, the first level of $\Gamma_{1,3}\subset \Gamma_2$, and $$(\omega\cdot 2+2, \omega\cdot 2+1, \omega\cdot 2, \omega +3, \omega+2) = (\omega\cdot 2+(2,1,0))\cat (\omega+(3,2))$$ is a (non-maximal) member of $\Lambda_{2, 3, 2}$, the second level of $\Gamma_{1,3}\subset \Gamma_2$.

If $s=\varnothing$ (resp. if $s$ is any maximal member of $\Lambda_{\xi, n, m}$ for some $1\leqslant m<n$), the set of proper extensions of $s$ which are on level $1$ (resp. $m+1$) is canonically identified with $\Omega_\xi$. Indeed, this set, say $U$,  of proper extensions of $s$ which lie on level $1$ (resp. level $m+1$) can be written as $$U=\{\omega^\xi(n-1)+t: t\in \Omega_\xi\}$$ $$(\text{resp.\ }U=\{s\cat (\omega^\xi(n-m-1)+t): t\in \Omega_\xi\}).$$   Here we are using the convention that if $t=(\zeta_i, u_i)_{i=1}^j\in \Omega_\xi=\Gamma_\xi D$, then for an ordinal $\zeta$, $\zeta+t=(\zeta+\zeta_i,  u_i)_{i=1}^j$. A   subsets of $\Omega_{\xi, n}$ which has the form of this set $U$ will be referred to as a \emph{unit}.  We define a function $\iota_{\xi,n}:\Omega_{\xi,n}\to \Omega_\xi$ by defining it on each unit. To that end, suppose $U$ is a unit and has one of the two forms given above. Then we define $\iota_{\xi,n}$ on $U$ by letting $$\iota_{\xi,n}(\omega^\xi(n-1)+t)=t$$ if $U=\{\omega^\xi(n-1)+t:t\in \Omega_\xi\}$ and $$\iota_{\xi,n}(s\cat (\omega^\xi(n-m-1)+t))=t$$ if $U=\{s\cat (\omega^\xi(n-m-1)+t): t\in \Omega_\xi\}$.

 We also recall the function $\mathbb{P}_\xi:\Omega_\xi\to [0,1]$ defined in \cite{C}. We let  $\mathbb{P}_0((0,U))=1$. If $\mathbb{P}_\xi$ has been defined,  for $t\in \Omega_{\xi, n}$, we let $\mathbb{P}_{\xi+1}(t)=\frac{1}{n}\mathbb{P}_\xi(\iota_{\xi, n}(t))$. If $\xi$ is a limit ordinal and $\mathbb{P}_\zeta$ has been defined for every $\zeta<\xi$,  and if $(t,\sigma)\in (\omega^\zeta+\Gamma_{\zeta+1})D$ for $\zeta<\xi$, we write $t=\omega^\zeta+t'$ for $t'\in \Gamma_{\zeta+1}$ and let $\mathbb{P}_\xi((t, \sigma))=\mathbb{P}_{\zeta+1}((t',\sigma))$.    It follows that for any ordinal $\xi$ and any maximal member $t$ of $\Omega_\xi$, $\sum_{s\preceq t}\mathbb{P}_\xi(s)=1$. 

For an ordinal $\xi$ and a natural number $n\in \nn$, we say an extended pruning $(\theta, e):\Omega_{\xi, n}\to \Omega_{\xi,n}$ is \emph{unit preserving} provided that for any unit $U$ of $\Omega_{\xi, n}$, the image $\theta(U)$ is contained within a single unit of $\Omega_{\xi, n}$, and if $s\prec t$, $s,t\in \Omega_{\xi, n}$, and $s,t$ lie on different levels of $\Omega_{\xi, n}$, $\theta(s), \theta(t)$ lie on different levels of $\Omega_{\xi, n}$.  We note that since in the case $n=1$, $\Omega_{\xi, n}=\Omega_\xi$ has only one unit and only one level, any extended pruning of $\Omega_{\xi, 1}$ into $\Omega_{\xi, 1}$ is unit preserving.

The main difficulty of adapting the argument from \cite{GKL} lies in the combinatorics.  As in the definition of the modulus $\rho_\xi$, we will need to check the norms of certain convex combinations.  In some arguments below, we will need one convex combination which simultaneously satisfies several different inequalities, and so it becomes necessary to check convex combinations with prescribed coefficients.  In the case that $\xi=0$, all sequences to be checked have legnth $n$ for some $n\in \nn$, and the prescribed coefficients needed are $1/n$.  However, the situation is more complicated when $\xi>0$, since the sequences in the higher order trees do not all have the same length. The $B$-trees and coefficients from $\mathbb{P}_\xi$ serve as an analogue to the repeated averages hierarchy from \cite{AMT}.  As we will see, in order to check the value of $\rho_\xi$, it is not required to check the supremum over all weakly null collections, but normally weakly null collections indexed by $\Omega_\xi$, where $D$ is a weak neighborhood basis at $0$ in $X$. Moreover, it is not required to check all convex combinations of all branches of the weakly null collection, but only those convex combinations the coefficients of which are given by the function $\mathbb{P}_\xi$.  This phenomenon actually has several applications, some of which are vital to this work.   Our main combinatorial tool of this work is the following, which is an improvement on the main combinatorial result from \cite{C}.  

\begin{theorem} For any ordinal $\xi$, any $n\in \nn$, any $\delta\in (0,1)$, and any function $f:\Pi\Omega_{\xi, n}\to [-1,1]$, there exist a unit preserving extended pruning $(\theta, e):\Omega_{\xi, n}\to \Omega_{\xi, n}$ and an $n$-tuple $(a_i)_{i=1}^n\subset [-1,1]$ such that for each $1\leqslant i\leqslant n$, \begin{enumerate}[(i)]\item for each $\Lambda_{\xi,n,i}\ni s\preceq t\in MAX(\Omega_{\xi, n})$, $f(\theta(s), e(t)) \geqslant a_i-\delta $, \item for each $t\in MAX(\Omega_{\xi, n})$, $$\sum_{\Lambda_{\xi,n,i}\ni s\preceq e(t)} \mathbb{P}_\xi(\iota_{\xi, n}(s))f(s, e(t)) \leqslant \delta/n+ a_i.$$   \end{enumerate}

In particular, for every $t\in MAX(\Omega_{\xi,n})$, $\sum_{s\preceq e(t)} \mathbb{P}_\xi(\iota_{\xi,n}(s))f(s,e(t))\leqslant \delta+\sum_{i=1}^n a_i$.

\label{favorite theorem}
\end{theorem}

\begin{remark}   

Note that in the case $n=1$, $\Omega_{\xi,n}=\Omega_\xi$ and $\iota_{\xi,n}$ is the identity on $\Omega_\xi$.  Therefore in the case that $n=1$ the conclusion of the theorem  above can be simplified to: There exist an extended pruning $(\theta,e):\Omega_\xi\to \Omega_\xi$ and a number $a\in [-1,1]$ such that  \begin{enumerate}[(i)]\item for each $(s,t)\in \Pi\Omega_\xi$, $f(\theta(s), e(t))\geqslant a-\delta$, \item for each $t\in MAX(\Omega_\xi)$, $$\sum_{s\preceq e(t)} \mathbb{P}_\xi(s)f(s, e(t))\leqslant a+\delta.$$  \end{enumerate}

By homogeneity, the assumption that the function $f$ maps into $[-1,1]$ may be relaxed.  The theorem holds just as well for any bounded function $f:\Pi\Omega_{\xi, n}\to \rr$, with the $a_i$ lying in $[-\|f\|_\infty, \|f\|_\infty]$.  

We also note that if we are willing to allow the numbers $a_i$ lie in $(-\infty,\|f\|_\infty]$ rather than in $[-\|f\|_\infty, \|f\|_\infty]$,  we may obtain for any bounded function $f:\Pi\Omega_{\xi, n}\to \rr$ and any $\delta>0$ some $a_i\in \rr$ as in the conclusion of the theorem with condition $(i)$ replaced by $f(\theta(s), e(t))\geqslant a_i$.  We will use this formulation throughout, although it is more convenient to prove the theorem as stated.  

Finally, we observe the following consequence of Theorem \ref{favorite theorem}.  If $b<b_1\in \rr$ and $f:\Pi\Omega_{\xi, n}\to \rr$ is a bounded function having the property that for every $t\in MAX(\Omega_{\xi, n})$, $\sum_{s\preceq t}\mathbb{P}_\xi(\iota_{\xi, n}(s))f(s,t)\geqslant b_1$, then there exist a unit preserving extended pruning $(\theta, e):\Omega_{\xi, n}\to \Omega_{\xi, n}$ and numbers $a_1, \ldots, a_n$ such that $\sum_{i=1}^n a_i>b$ and for any $(s,t)\in \Pi\Omega_{\xi, n}$ such that $s\in \Lambda_{\xi, n, i}$, $f(\theta(s), e(t))\geqslant a_i$.  Indeed, we may choose $\delta>0$ such that $(n+1)\delta+b<b_1$ and obtain a unit preserving extended pruning $(\theta, e):\Omega_{\xi, n}\to \Omega_{\xi, n}$ and numbers $a_1', \ldots ,a_n'$ as in the conclusion of the theorem.  Taking $a_i=a_i'-\delta$ and noting that, for any $t\in MAX(\Omega_{\xi, n})$, $$b_1\leqslant \sum_{s\preceq e(t)}\mathbb{P}_\xi(\iota_{\xi, n}(s))f(s, e(t))\leqslant \delta+\sum_{i=1}^n a_i'\leqslant (n+1)\delta +\sum_{i=1}^n a_i$$ yields that $b< \sum_{i=1}^n a_i$.

\end{remark}

We will need the following results from \cite{CIllinois} and \cite{C} before completing the proof of Theorem \ref{favorite theorem}.  In the following proposition, for $t\in MAX(\Lambda_{\xi, n+1, 1})$, $B_t=\{s\in \Omega_{\xi, n+1}: t<s\}$. 

\begin{proposition} \begin{enumerate}[(i)]\item For any ordinal $\xi$,  $n\in \nn$, any finite set $S$, any $t\in MAX(\Lambda_{\xi, n+1, 1})$, and any function $f:\Pi\Omega_{\xi, n+1}\to S$, there exists a unit preserving extended pruning $(\theta,e):B_t\to B_t$ such that for every $s\preceq t$, $f(s, e(\cdot))$ is constant on $MAX(B_t)$.       \item  For any ordinal $\xi$, any natural number  $n\in \nn$, any bounded function  $f:\Pi\Omega_{\xi, n+1}\to \rr$, any $t\in MAX(\Lambda_{\xi, n+1, 1})$, and any $\delta>0$, there exists a unit preserving extended pruning $(\theta, e):B_t\to B_t$ such that for every $s\preceq t$, $$\text{\emph{diam}}\{f(s, e(u)):(t,u)\in \Pi\Omega_{\xi, n+1}\}<\delta.$$ \item For any ordinal $\xi$, $n\in \nn$, any finite set $S$, and any function $f:\Lambda_{\xi, n,1}\to S$, there exists an extended pruning $(\theta, e):\Lambda_{\xi, n,1}\to \Lambda_{\xi, n,1}$ such that $f(\theta(\cdot), e(\cdot))$ is constant on $MAX(\Lambda_{\xi, n,1})$.    \item For any ordinal $\xi$, $n\in \nn$, any bounded $f:\Pi\Lambda_{\xi, n,1}\to \rr$, and any $\delta>0$, there exists an extended pruning $(\theta, e):\Lambda_{\xi, n, 1}\to \Lambda_{\xi, n,1}$ such that $$\text{\emph{diam}}\{f(\theta(s),e(t)): (s,t)\in \Pi\Lambda_{\xi, n,1}\}<\delta.$$   \end{enumerate}

\label{recall}
\end{proposition}

\begin{proof} Item $(i)$ is contained within the proof of \cite[Lemma $6.2$]{C} and item $(iii)$ in \cite[Proposition $4.2$]{CIllinois}.  Items $(ii)$ and $(iv)$ follow from applying $(i)$ and $(iii)$, respectively.  Indeed, after fixing $f:\Pi\Omega_{\xi,n+1}\to \rr$, we may fix a finite partition $S$ of $[-\|f\|_\infty, \|f\|_\infty]$ by sets of diameter less than $\delta$ and define $g:\Pi\Omega_{\xi, n+1}\to S$ by letting $g(s,t)$ be the member of $S$ containing $f(s,t)$.  

\end{proof}

We will also need the following easy observation.

\begin{proposition} Fix an ordinal $\xi$ and a directed set $D$. Let $\Omega_{\xi, n}=\Gamma_{\xi, n} D$ for each $n\in \nn$.   For any finite, non-empty subset $A$ of $\nn$ and any $\max A\leqslant n\in \nn$, there exists an extended pruning $(\theta,e):\Omega_{\xi, |A|}\to \Omega_{\xi, n}$ such that $\theta(\Omega_{\xi, |A|})\subset \cup_{i\in A} \Lambda_{\xi, n, i}$.   

\label{easy}
\end{proposition}

\begin{proof} We may define a pruning $\theta$ to have the desired property and let $e$ be any extension.  Such an extension exists since all trees involved are well-founded. Thus we do not need to specify the function $e$.  We induct on $|A|$.  If $|A|=1$, say $A=\{k\}$,  and $n\geqslant k$, let $U$ be any unit on level 
$k$ of $\Omega_{\xi, n}$.    Recall that $\iota_{\xi, n}|_U:U\to \Omega_\xi=\Omega_{\xi, |A|}$ is the canonical identification, and we may take $\theta= \iota_{\xi,n}|_U^{-1}$ in this case.   It is easy to check that this is a pruning.

Next, assume $|A|=r+1$ and $n\geqslant \max A$.   Let $k=\min A$ and let $A_1=\{j-k: j\in A\setminus\{k\}\}$.   Let $U$ be any unit on level $k$ of $\Omega_{\xi, n}$ and define $\theta|_{\Lambda_{\xi, |A|,1}}=\iota_{\xi, n}|_U^{-1}\circ \iota_{\xi, |A|}$.   For each $t\in MAX(\Lambda_{\xi, |A|, 1})$, let $B_t$ denote the collection of proper extensions of $t$ in $\Omega_{\xi, |A|}$. Note that $t$ is mapped by $\theta$ to some $s_t\in MAX(\Lambda_{\xi, n, k})$.  Let $C_t$ denote the proper extensions of $s_t$ in $\Omega_{\xi,n}$.    Note that $B_t$ is identifiable with $\Omega_{\xi, r}$ via a function $j_t:\Omega_{\xi, r}\to B_t$ given by  $$((\omega^\xi(r-1)+t_1)\cat\ldots \cat(\omega^\xi(r-i)+t_i), \sigma)\underset{j_t}{\mapsto}t\cat ((\omega^\xi(r-1)+ t_1)\cat\ldots \cat(\omega^\xi(r-i)+t_i), \sigma),$$   and this is a pruning.  Similarly, $C_t$ is identifiable with $\Omega_{\xi, n-k}$ via $i_t:\Omega_{\xi, n-k}\to C_t$ given by $$((\omega^\xi(n-k-1)+t_1)\cat\ldots \cat(\omega^\xi(n-k-i)+t_i), \sigma)\underset{i_t}{\mapsto} s_t\cat (\omega^\xi(n-k-1)+t_1)\cat \ldots \cat(\omega^\xi(n-k-i)+t_i), \sigma).$$  Note that $i_t$ maps $\Lambda_{\xi, n-k, i}$ into $\Lambda_{\xi, n, i+k}$.  Apply the inductive hypothesis to $n-k$ and $A_1$ to obtain a pruning $\theta_1:\Omega_{\xi, r}\to \Omega_{\xi, n-k}$ such that $\theta_1(\Omega_{\xi, r})\subset \cup_{i\in A_1} \Lambda_{\xi, n, i}$.  Now define $\theta$ on $B_t$ by setting $\theta= i_t\circ \theta_1\circ j_t^{-1}$. It is easy to see that this is a pruning, and $$\theta(B_t)\subset i_t(\cup_{i\in A_1}\Lambda_{\xi, n-k,i})\subset \cup_{i\in A_1}\Lambda_{\xi, n, i+k}\subset \cup_{i\in A}\Lambda_{\xi, n, i}.$$

\end{proof}

\begin{proof}[Proof of Theorem \ref{favorite theorem}] We prove the result by induction on $\ord\times \nn$ ordered lexicographically.  We first prove the $(0,1)$ case.  Note that $\Omega_{1,0}=\Omega_0=\{(0,U):U\in D\}$ and $\Pi\Omega_{\xi,n}= \{(t,t): t\in \Omega_0\}$ in this case.  By compactness of $[-1,1]$, there exists $a\in [-1,1]$ such that $$M:=\{U\in D:|a-f((0,U), (0,U))|<\delta\}$$ is cofinal in $D$. We define $\theta,e:\Omega_0\to \Omega_0$ by  letting $$\theta((0,U))=e((0,U))=(0,V),$$ where $V\in M$ is such that $U\leqslant_D V$.  It is straightforward to check that the conclusions are satisfied in this case.

Next, suppose that $\xi$ is a limit ordinal and the conclusion holds for every pair $(\zeta, 1)$ with $\zeta<\xi$.   For each $\zeta<\xi$, let $\Theta_{\zeta+1}=(\omega^\zeta+\Gamma_{\zeta+1})D$, so that $\Omega_{\xi,1}=\Omega_\xi=\cup_{\zeta<\xi}\Theta_{\zeta+1}$.  Note that $(\Theta_{\zeta+1}, \mathbb{P}_\xi)$ is canonically identifiable with $(\Omega_{\zeta+1}, \mathbb{P}_{\zeta+1})$.  Therefore for each $\zeta<\xi$, we may apply the inductive hypothesis to $f|_{\Pi\Theta_{\zeta+1}}$ to obtain $a_\zeta\in [-1,1]$ and an extended pruning $(\theta_\zeta, e_\zeta):\Theta_{\zeta+1}\to \Theta_{\zeta+1}$ such that for all $(s,t)\in \Pi\Theta_{\zeta+1}$, $f(\theta_\zeta(s), e_\zeta(t))\geqslant a_\zeta-\delta/2$, and for each $t\in MAX(\Theta_{\zeta+1})$, $$\sum_{s\preceq e_\zeta(t)}\mathbb{P}_\xi(s) f(s, e_\zeta(t))\leqslant \delta/2+a_\zeta.$$  There exists $a\in [-1,1]$ such that the set $$M=\{\zeta<\xi: |a_\zeta-a|<\delta/2\}$$ is cofinal in $[0,\xi)$.  For every $\zeta<\xi$, fix $\eta_\zeta\in M$ such that $\zeta\leqslant \eta_\zeta$ and fix an extended pruning $(\varphi_\zeta, f_\zeta):\Theta_{\zeta+1}\to \Theta_{\eta_\zeta+1}$.   Let $\theta|_{\Theta_{\zeta+1}}=\theta_{\eta_\zeta}\circ\varphi_\zeta$ and let $e|_{MAX(\Theta_{\zeta+1})}=e_{\eta_\zeta}\circ f_\zeta$.  It is straightforward to check that the conclusions are satisfied in this case.

Next, suppose that for some ordinal $\xi$ and every $n\in \nn$, the result holds for the pair $(\xi, n)$.   Fix $f:\Pi\Omega_{\xi+1, 1}\to [-1,1]$.  Apply the inductive hypothesis to $f|_{\Omega_{\xi, n}}$ to obtain a unit preserving extended pruning $(\theta_n, e_n):\Omega_{\xi, n}\to \Omega_{\xi, n}$ and $(a^n_i)_{i=1}^n\subset [-1,1]$ such that for every $(s,t)\in \Pi\Omega_{\xi, n}$, $f(\theta_n(s), e_n(t))\geqslant a^n_i-\delta/2$ and for every $t\in MAX(\Omega_{\xi, n})$, $$\sum_{s\preceq e_n(t)}\mathbb{P}_\xi(\iota_{\xi, n}(s)) f(s, e_n(t)) \leqslant \delta/2+\sum_{i=1}^n a_i^n.$$  Fix $n_1<n_2<\ldots$ and $a\in [-1,1]$ such that $$\frac{1}{n_j}\sum_{i=1}^{n_j} a_i^{n_j} \underset{j}{\to} a.$$ By passing to another subsequence, we may assume that for every $j\in \nn$, $\frac{1}{n_j}\sum_{i=1}^{n_j} a_i^{n_j}\leqslant a+\delta/2$.     For each $j\in \nn$, let $$A_j=\{i\leqslant n_j: a^{n_j}_i>a-\delta/2\}.$$  We claim that $(|A_j|)_{j=1}^\infty$ is unbounded.  Indeed, if $|A_j|\leqslant k$ for all $j\in\nn$, then $$a\underset{j}{\leftarrow} \frac{1}{n_j}\sum_{i=1}^{n_j} a_i^{n_j} \leqslant \frac{1}{n_j} \sum_{i\in A_j}1 + \frac{1}{n_j}\sum_{A_j\not\ni \text{\ } i\leqslant n_j} a_i^{n_j} \leqslant \frac{k}{n_j}+(a-\delta/2)\cdot\frac{n_j-|A_j|}{n_j}\underset{j}{\to} a-\delta/2,$$ a contradiction.   By passing once more to a subsequence, we may assume that for all $j\in \nn$, $j\leqslant |A_j|$.   For every $j\in \nn$, by Proposition \ref{easy}, we may choose a unit preserving extended pruning $(\varphi_j, f_j):\Omega_{\xi, j}\to \Omega_{\xi, n_j}$ such that $\varphi_j(\Omega_{\xi, j})\subset \cup_{i\in A_j} \Lambda_{\xi, n_j, i}$.   Then define $\theta:\Omega_{\xi+1}\to \Omega_{\xi+1}$ and $e:MAX(\Omega_{\xi+1})\to MAX(\Omega_{\xi+1})$ by $\theta|_{\Omega_{\xi, j}}= \theta_{n_j}\circ\varphi_j$ and $e|_{MAX(\Omega_{\xi, j})}=e_{n_j}\circ f_j$.  For every $(s,t)\in\Omega_{\xi+1}$, there exist $j\in \nn$ and $i\in A_j$ such that $\theta(s)\in \Lambda_{\xi, n_j, i}$, so that $$f(\theta(s), e(t)) \geqslant a^{n_j}_i-\delta/2>a-\delta/2-\delta/2=a-\delta.$$   For any $t\in MAX(\Omega_{\xi+1})$, if $e(t)\in \Omega_{\xi, n_j}$, then for any $s\preceq e(t)$, $\mathbb{P}_{\xi+1}(s)= \frac{1}{n_j}\mathbb{P}_\xi(\iota_{\xi, n_j}(s))$.   Therefore \begin{align*} \sum_{s\preceq e(t)}\mathbb{P}_{\xi+1}(s)f(s, e(t)) &  =\frac{1}{n_j}\sum_{s\preceq t}\mathbb{P}_\xi(\iota_{\xi, n_j}(s))f(s, e(t)) \leqslant \frac{1}{n_j}\Bigl(\delta/2+\sum_{i=1}^{n_j} a_i^{n_j}\Bigr) \\ & \leqslant \delta/2n_j+\frac{1}{n_j}\sum_{i=1}^{n_j} a_i^{n_j} \leqslant \delta/2+\delta/2+a=a+\delta. \end{align*}

Last, assume that for some ordinal $\xi$ and some $n\in \nn$, the result holds for every pair $(\xi, k)$ with $1\leqslant k\leqslant n$.  Fix a function $f:\Pi\Omega_{\xi, n+1}\to [-1,1]$.   For each $t\in MAX(\Lambda_{\xi, n, 1})$, let $$B_t=\{s\in \Omega_{\xi, n+1}: t\prec s\}.$$  Let $\kappa:B_t\to \Omega_{\xi, n}$ be the canonical identification. Note that for any $s\in B_t$, $\mathbb{P}_\xi(\iota_{\xi, n+1}(s))=\mathbb{P}_\xi(\iota_{\xi, n}(\kappa(s)))$.     We may therefore identify $f|_{\Pi B_t}$ with a function on $\Pi\Omega_{\xi, n}$ to obtain an $n$-tuple $(\alpha^t_i)_{i=2}^{n+1}\subset [-1,1]$ and a unit preserving extended pruning $(\theta_t, e_t):B_t\to B_t$ such that for every $(s,s_1)\in \Pi B_t$, if $s\in \Lambda_{\xi, n+1, i}$, $f(\theta_t(s), e_t(s_1)) \geqslant \alpha^t_i-\delta/4$, and such that for every $t_1\in MAX(B_t)$, $$\sum_{t\prec s\preceq t_1} \mathbb{P}_\xi(\iota_{\xi, n+1}(s))f(s, e_t(t_1)) \leqslant \delta/4+\sum_{i=2}^{n+1} \alpha^t_i.$$ Apply Proposition \ref{recall}$(ii)$ to choose another unit preserving extended pruning $(\varphi_t, f_t):B_t\to B_t$ such that for every $s\preceq t$, $$\text{diam}\{f(s, e_t\circ f_t(t_1)): (t, t_1)\in \Pi\Omega_{\xi, n+1}\}<\delta/4.$$  Let $T\subset [-1,1]^n$ be a finite $\delta/4$-net in $[-1,1]^n$ (endowed with the $\ell_1^n$ distance) and for each $t\in MAX(\Lambda_{\xi, n,1})$, fix $(a^t_i)_{i=2}^{n+1}\in T$ such that $\|(a_i^t)_{i=2}^{n+1}-(\alpha_i^t)_{i=2}^{n+1}\|_{\ell_1^n}<\delta/4$.   Note that for each $t\in MAX(\Lambda_{\xi, n+1, 1})$, for each $(s, s_1)\in \Pi B_t$, if $s\in \Lambda_{\xi, n+1, i}$, $$f(\theta_t\circ \varphi_t (s), e_t\circ f_t(s_1))\geqslant \alpha^t_i-\delta/4\geqslant a^t_i-\delta/2,$$ and if $t_1\in MAX(B_t)$, $$\sum_{t\prec s\preceq t_1} \mathbb{P}_\xi(\iota_{\xi, n+1}(s))f(s, e_t\circ f_t(t_1))\leqslant \delta/4+\sum_{i=2}^{n+1}\alpha^t_i \leqslant \delta/2+\sum_{i=2}^{n+1}a_i^t.$$   Define $F:\Pi\Lambda_{\xi, n+1, 1}\to [-1,1]$ by letting $$F(s,t)=\inf\{f(s, e_t\circ f_t(t_1)): (t, t_1)\in \Pi\Omega_{\xi, n+1}\}.$$   Note that by our choice of $(\varphi_t, f_t)$, for any $(s,t)\in \Pi\Lambda_{\xi, n+1,1}$, $$\sup\{f(s, e_t\circ f_t(t_1)): (t, t_1)\in \Pi\Omega_{\xi, n+1}\} \leqslant \delta/4+F(s,t).$$   Identifying $(\Lambda_{\xi, n+1, 1}, \mathbb{P}_\xi\circ \iota_{\xi, n})$ with $(\Omega_\xi, \mathbb{P}_\xi)$, we may apply the inductive hypothesis to find an extended pruning $(\theta', e'):\Lambda_{\xi, n+1, 1}\to \Lambda_{\xi, n+1, 1}$ and a number $a_1\in [-1,1]$ such that for every $(s,t)\in \Pi\Lambda_{\xi, n+1, 1}$, $F(\theta'(s), e'(t))\geqslant a_1-\delta/4$ and such that for every $t\in MAX(\Lambda_{\xi, n+1, 1})$, $$\sum_{s\preceq e'(t)} \mathbb{P}_\xi(\iota_{\xi, n+1}(s)) F(s, e'(t)) \leqslant \delta/4+a_1.$$  Then for any $t\in MAX(\Lambda_{\xi, n+1, 1})$ and any maximal extension $t_1$ of $t$, we deduce that $$\sum_{s\preceq e'(t)} \mathbb{P}_\xi(\iota_{\xi, n+1}(s))f(s, e_{e'(t)}\circ f_{e'(t)}(t_1)) \leqslant \sum_{s\preceq e'(t)} \mathbb{P}_\xi(\iota_{\xi, n+1}(s)) (\delta/4+F(s, e'(t))) \leqslant \delta/2+a_1.$$ Last, applying Proposition \ref{recall}$(iii)$, we may choose another extended pruning $(\theta'', e''):\Lambda_{\xi, n+1, 1}\to \Lambda_{\xi, n+1, 1}$ and $(a_i)_{i=2}^{n+1}\in T$ such that for every $t\in MAX(\Lambda_{\xi, n+1, 1})$, $(a_i)_{i=2}^{n+1}=(a^{e'\circ e''(t)}_i)_{i=2}^{n+1}$.  For each $t\in MAX(\Lambda_{\xi, n+1, 1})$, let $\kappa_t:B_t\to B_{e'\circ e''(t)}$ be the canonical identification.   Define $\theta:\Omega_{\xi, n+1}\to \Omega_{\xi, n+1}$ and $e:MAX(\Omega_{\xi, n+1})\to MAX(\Omega_{\xi, n+1})$ by letting $$\theta|_{\Lambda_{\xi, n+1, 1}}= \theta\circ \theta'',$$ $$\theta|_{B_t}= \theta_{e'\circ e''(t)}\circ \varphi_{e'\circ e''(t)}\circ \kappa_t,$$ $$e|_{MAX(B_t)}=e_{e'\circ e''(t)}\circ f_{e'\circ e''(t)}\circ \kappa_t.$$  We last show that the conclusions are satisfied in this case.   First fix $(s, s_1)\in \Pi\Omega_{\xi, n+1}$ such that $s\in \Lambda_{\xi, n+1, 1}$.  Let $t$ be the member of $MAX(\Lambda_{\xi, n+1, 1})$ such that $s\preceq t\prec s_1$.  Let $s'=\theta''(s)$ and $t'=e''(t)$.   Since $e(s_1)\in B_{e'\circ e''(t)}=B_{e'(t')}$, $$(e'(t'), e(s_1))\in \Pi\Omega_{\xi, n+1},$$ and \begin{align*} a_1-\delta & < F(\theta'(s'), e'(t')) =\inf \{f(\theta'\circ \theta''(s), e_{e'(t')}\circ f_{e'(t')}(t_1)): (e'(t'), t_1)\in \Pi\Omega_{\xi, n+1}\} \\ & \leqslant f(\theta'\circ \theta''(s), e(s_1))= f(\theta(s), e(s_1)).\end{align*}

 Next, fix $(s,s_1)\in \Pi\Omega_{\xi, n+1}$ such that $s\in \Lambda_{\xi, n+1, i}$ for some $2\leqslant i\leqslant n+1$.   Fix $t\in MAX(\Lambda_{\xi, n+1, 1})$ such that $s\in B_t$. Let $s'=\varphi_{e'\circ e''(t)}\circ \kappa_t(s)$ and $s_1'=f_{e'\circ e''(t)}\circ \kappa_t(s_1)$.   Then \begin{align*} a_i-\delta & = a_i^{e'\circ e''(t)}-\delta < f(\theta_{e'\circ e''(t)}(s'), e_{e'\circ e''(t)}(s_1')) = f(\theta(s), e(s_1)).\end{align*}

 Finally, fix $t_1\in MAX(\Omega_{\xi, n+1})$ and let $t$ be the member of $MAX(\Lambda_{\xi, n+1, 1})$ such that $t\prec t_1$.  Recall that  $$\sum_{s\preceq e'\circ e''(t)}\mathbb{P}_\xi(\iota_{\xi, n+1}(s))f(s, e(t_1))\leqslant \delta/2+a_1$$ and $$\sum_{e'\circ e''(t)\prec s\preceq e(t_1)} \mathbb{P}_\xi(\iota_{\xi, n+1}(s))f(s, e(t_1))<\delta/2+\sum_{i=2}^{n+1} a_i^{e'\circ e''(t)} = \delta/2+\sum_{i=2}^{n+1}a_i.$$   Thus $$\sum_{s\preceq t}\mathbb{P}_{\xi, n+1}(s) f(s, e(t_1)) \leqslant \delta+\sum_{i=1}^{n+1} a_i$$ for every $t_1\in MAX(\Omega_{\xi, n+1})$.

\end{proof}

We have another combinatorial result which extends work from \cite{CIllinois}.  

\begin{lemma} Fix an ordinal $\xi$.  For any natural number, any finite set $\mathcal{P}$, and any function $f:\Pi\Omega_{\xi, n}\to \mathcal{P}$, there exists a unit preserving extended pruning $(\theta, e):\Omega_{\xi, n}\to \Omega_{\xi, n}$ and an $n$-tuple $(p_i)_{i=1}^n\in \mathcal{P}^n$ such that for any $(s,t)\in \mathcal{P}$, if $s\in \Lambda_{\xi, n, i}$, $f(\theta(s), e(t))=p_i$.  
\label{combinatorial}
\end{lemma}

\begin{proof} We induct on $n$.  The $n=1$ case follows from \cite[Proposition $4.2$]{CIllinois}.  Next, fix a finite set $\mathcal{P}$ and a function $f:\Pi\Omega_{\xi, n+1}\to \mathcal{P}$.  For every $t\in MAX(\Lambda_{\xi, n+1, 1})$, let $B_t$ denote the proper extensions of $t$.  Identifying $B_t$ with $\Omega_{\xi, n}$, we may find an $n$-tuple $(p^t_i)_{i=2}^{n+1}\in \mathcal{P}^n$ and a unit preserving extended pruning $(\theta_t, e_t):B_t\to B_t$ such that for any $t\prec s\preceq t_1\in MAX(B_t)$, if $s\in \Lambda_{\xi, n+1, i}$, $f(\theta_t(s), e_t(t_1))=p^t_i$.  Let $\overline{p}_t=(p^t_i)_{i=2}^{n+1}$.  

Next, by Proposition \ref{recall}$(i)$, we may find another unit preserving extended pruning $(\varphi_t, f_t):B_t\to B_t$ such that for every $s\preceq t$ and any maximal extensions $t_1, t_2$ of $t$, $f(s, e_t\circ f_t(t_1))=f(s, e_t\circ f_t(t_2))$.  Define $F:\Pi\Lambda_{\xi, n+1, 1}\to \mathcal{P}$ by letting $F(s, t)$ be the common value of $f(s, e_t\circ f_t(t_1))$ as $t_1$ ranges through all maximal extensions of $t$.  Fix an extended pruning $(\theta', e'):\Lambda_{\xi, n+1, 1}\to \Lambda_{\xi, n+1,1}$ and $p_1\in \mathcal{P}$ such that for every $(s,t)\in \Pi\Lambda_{\xi, n+1, 1}$, $f(\theta'(s), e'(t))=p_1$.   Next, using Proposition \ref{recall}$(i)$ again, we may fix another extended pruning $(\varphi, f):\Lambda_{\xi, n+1, 1}\to \Lambda_{\xi, n+1,1}$ and $(p_i)_{i=2}^{n+1}\in \mathcal{P}^n$ such that for every $t\in MAX(\Lambda_{\xi, n+1, 1})$, $\overline{p}_{e'\circ f(t)}=(p_i)_{i=2}^{n+1}$.  For each $t\in MAX(\Lambda_{\xi, n+1, 1})$, let $j_t:B_t\to B_{e'\circ f(t)}$ be the canonical identification.  We let $\theta|_{\Lambda_{\xi, n+1,1}}=\theta'\circ \varphi$, $\theta|_{B_t}=\theta_{e'\circ f(t)}\circ \varphi_{e'\circ f(t)}\circ j_t$ and $e|_{MAX(B_t)}=e_{e'\circ f(t)}\circ f_{e'\circ f(t)}\circ j_t$.

\end{proof}

\begin{proposition} Let $\varnothing\neq K\subset X^*$ be $w^*$-compact and let $\xi$ be an ordinal. Let $D$ be a weak neighborhood basis at $0$ in $X$ directed by reverse inclusion and let $\Omega_\xi=\Gamma_\xi D$.   \begin{enumerate}[(i)]\item Suppose $\sigma, \ee,a>0$, $a/\sigma>\ee$, $(x_\lambda)_{\lambda\in D}\subset \sigma B_X$ is a weakly null net, and $(x^*_\lambda)_{\lambda\in D}\subset K$ is such that for every $\lambda$, $\text{\emph{Re}\ }x^*_\lambda(x_\lambda)\geqslant a$.  Then any $w^*$-cluster point of $(x^*_\lambda)_{\lambda\in D}$ lies in $s_\ee(K)$. \item If $\sigma, \ee, a>0$, $a/\sigma>\ee$,  $(x_t)_{t\in \Omega_\xi}\subset \sigma B_X$ is normally weakly null, and $(x^*_t)_{t\in MAX(\Omega_\xi)}\subset K$ is such that for every $s\preceq t\in MAX(\Omega_\xi)$, $\text{\emph{Re}\ }x^*_t(x_s)\geqslant a$, then $$s_\ee^{\omega^\xi}(K)\cap \overline{\{x^*_t:t\in MAX(\Omega_\xi)\}}^{w^*}\neq \varnothing.$$ \item For $N\in \nn$, $\sigma>0$, and $\ee_1, \ldots, \ee_N\in \rr$, in order that $s_{\xi, \ee_1}\ldots s_{\xi, \ee_N}(K)\neq \varnothing$, it is sufficient that there exist a normally weakly null $(x_t)_{t\in \Omega_{\xi, N}}\subset \sigma B_X$, numbers $a_1, \ldots, a_N$ with $a_i/\sigma>\ee_i$ if $\ee_i> 0$, and for every $t\in MAX(\Omega_{\xi, N})$ some $x^*_t\in K$ such that for every $s\preceq t$, if $s\in \Lambda_{\xi, N, i}$ and if $\ee_i>0$,  $\text{\emph{Re}\ }x^*_t(x_s)\geqslant a_i$. \end{enumerate}

\label{characterize}
\end{proposition}

\begin{rem}\upshape In $(iii)$, we are witnessing the non-emptiness of a certain derived set of a given set using weakly null trees of prescribed order in $X$. The exact origin of this method is perhaps difficult to specify. The germ of this idea appears in \cite[Lemma $3.4$]{Lancien1}. This particular result can be thought of as a transfinite version of \cite[Lemma $4.3$]{GKL} combined with the characterization of the Szlenk index of an arbitrary weak$^*$-compact set given in \cite{CIllinois}, which was itself a generalization of the characterization of the Szlenk index of $B_{X^*}$ when $X$ is separable and not containing $\ell_1$ given in \cite{AJO}.

\end{rem}

\begin{proof}$(i)$ Suppose $x^*=w^*-\lim_{\lambda\in D_0} x^*_\lambda\in K$, where $(x^*_\lambda)_{\lambda\in D_0}$ is a subnet of $(x^*_\lambda)_{\lambda\in D}$.  Then $$\underset{\lambda\in D_0}{\lim\inf} \|x^*_\lambda-x^*\| \geqslant \underset{\lambda\in D_0}{\lim\inf\ }\text{Re\ }(x^*_\lambda-x^*)(x_\lambda)/\sigma\geqslant a/\sigma>\ee,$$  whence $s\in s_\ee(K)$.

$(ii)$ Let $L=\overline{\{x^*_t: t\in MAX(\Omega_\xi)\}}^{w^*}\subset K$. If $\xi=0$, the result follows from $(i)$, since $(x_t)_{t\in \Omega_0}$ is a single weakly null net. Assume $\xi>0$.  By $w^*$-compactness, it is sufficient to show that for every $\zeta<\omega^\xi$, $L\cap s^\zeta_\ee(K)\neq \varnothing$.  Define $h:\Omega_\xi \to [0, \omega^\xi)$ by letting $h(t)=\max\{\zeta< \omega^\xi: t\in \Omega_\xi^\zeta\}$.    We claim that for any $0\leqslant \zeta<\omega^\xi$, if $h(t)\geqslant \zeta$, there exists $x^*\in s^\zeta_\ee(K)\cap L$ such that for every $s\preceq t$, $\text{Re\ }x^*(x_s)\geqslant a$.  We prove this claim by induction on $\zeta$.  If $\zeta=0$, fix any $t\in \Omega_\xi$ and any maximal extension $u$ of $t$. By hypothesis, $x^*_u\in K\cap L$ is such that for every $s\preceq u$, and therefore every $s\preceq t$, $\text{Re\ }x^*_u(x_s)\geqslant a$.  Next, assume the result holds for $\zeta$ and $h(t)\geqslant \zeta+1$.  This means $t$ admits an extension, and therefore an extension by $1$, in $\Omega_\xi^\zeta$.  Then there exists $\gamma\in [0, \omega^\xi)$ such that for every $U\in D$, $t_U:= t\cat (\gamma, U)\in \Omega_\xi^\zeta$.  This means $h(t_U)\geqslant \zeta$, and by the inductive hypothesis, there exists $z^*_U\in s^\zeta_\ee(K)\cap L$ such that for every $s\preceq t_U$, $\text{Re\ }z^*_U(x_s)\geqslant a$.  Let $x^*$ be any $w^*$-limit of a subnet of $(z^*_U)_{U\in D}$.  By applying $(i)$ to the nets $(x_{t_U})_{U\in D}$ and $(z^*_U)_{U\in D}$, we deduce that $x^*\in s^{\zeta+1}_\ee(K)\cap L$.  Last, suppose $\zeta<\omega^\xi$ is a limit ordinal and the result holds for every ordinal less than $\zeta$.  Suppose $h(t)\geqslant \zeta$.  Therefore $h(t)>\eta$ for every $0\leqslant \eta<\zeta$, and by the inductive hypothesis, there exists $z^*_\eta\in s^\eta_\ee(K)\cap L$ such that for every $s\preceq t$, $\text{Re\ }z^*_\eta(x_s)\geqslant a$.   Then we may let $x^*$ be any $w^*$-limit of a subnet of $(z^*_\eta)_{\eta<\zeta}$.

$(iii)$ We claim that for every $0\leqslant i\leqslant N$, for every $t\in MAX(\Lambda_{\xi, N, i})$, there exists $x^*\in s_{\xi, \ee_{i+1}}\ldots s_{\xi, \ee_N}(K)$ such that for each $0< j\leqslant i$ such that $\ee_j>0$ and each $s\preceq t$ with $s\in \Lambda_{\xi, N, j}$, $\text{Re\ }x^*(x_s) \geqslant a_j$. Here, we agree to the convention that $\Lambda_{\xi, N, 0}=\{\varnothing\}$.  In the case that $i=N$, it is understood that $s_{\xi, \ee_{i+1}}\ldots s_{\xi, \ee_N}(K)=K$.  We induct on $N-i$. The case $i=N$ is the hypothesis.   Assume the result holds for some $i$ such that $0<i\leqslant N$. Fix $t\in MAX(\Lambda_{\xi, N, i-1})$ and let  $U$ denote the set of proper extensions of $t$ which lie in $\Lambda_{\xi, N, i}$. Note that $U$ is a unit and therefore canonically identifiable with $\Omega_\xi$.  For each maximal member $s$ of $U$, there exists $z^*_s\in s_{\xi, \ee_{i+1}}\ldots s_{\xi, \ee_N}(K)$ as in the claim.  Let $L=\overline{\{z^*_s: s\in MAX(U)\}}^{w^*}$.  If $\ee_i\leqslant 0$, we may let $x^*$ be any member of $L$.  If $\ee_i>0$, we apply $(ii)$ to the set $s_{\xi, \ee_{i+1}}\ldots s_{\xi, \ee_N}(K)$ to deduce the existence of some $x^*\in s_{\xi, \ee_i}\ldots s_{\xi,\ee_N}(K)\cap L$. Since $x^*\in L$, for any $0<j \leqslant i-1$ such that $\ee_j>0$ and each $s\preceq t$ such that $s\in \Lambda_{\xi, N, j}$, $\text{Re\ }x^*(x_s)\geqslant a_j$.  This is because $x^*\in L$ and $L$ is the $w^*$-closure of a set all of the members of which satisfy these inequalities.   This finishes the induction.  The case that $i=0$ gives the non-emptiness of $s_{\xi, \ee_1}\ldots s_{\xi, \ee_N}(K)$.

\end{proof}

Given $N\in \nn$, an ordinal $\xi$, and a collection of vectors $(x_t)_{t\in \Omega_{\xi, N}}$, for each $t\in MAX(\Omega_{\xi, N})$, we let $(z^t_i)_{i=1}^N$ denote the convex block of $(x_s:s\preceq t)$ given by $z^t_i=\sum_{t_{i-1}\prec s\preceq t_i}\mathbb{P}_\xi(\iota_{\xi, N}(s))x_s$.  Here, $\varnothing=t_0\prec \ldots \prec t_N=t$ are such that $t_i\in MAX(\Lambda_{\xi, N, i})$ for each $1\leqslant i\leqslant N$.  This notation does not reference the underlying set $(x_t)_{t\in \Omega_{\xi, N}}$, but this will cause no confusion.  We will call each sequence of the type $(z_i^t)_{i=1}^N$ a \emph{special convex block} of the collection $(x_t)_{t\in \Omega_{\xi, N}}$. 

Given an operator $A:X\to Y$ and $\sigma>0$, let $N_\xi(\sigma)$ denote the smallest $N\in \nn\cup \{0\}$ such that there exists a normally weakly null $(x_t)_{t\in \Omega_{\xi,N+1}}\subset \sigma B_X$ such that for every $t\in MAX(\Omega_{\xi, N+1})$, $\|\sum_{i=1}^N Az_i^t\|>1$.   We write $N_\xi(\sigma)=\infty$ if no such $N$ exists. If there is a need to specify the operator, we write $N_\xi(\sigma;A)$.  If $A=0$, then $N_\xi\equiv \infty$, and otherwise $N_\xi(\sigma)>0$ for $0<\sigma\leqslant \|A\|^{-1}$.   Given $n\in \nn$, $y\in Y$, and $\sigma\in (0,1)$, let $\varrho_\xi(y, n, \sigma)$ denote the infimum of all $\lambda\in \rr$ such that for every normally weakly null collection $(x_t)_{t\in \Omega_{\xi, n}}\subset \sigma B_X$, there exists $t\in MAX(\Omega_{\xi, n})$ such that $\|y+\sum_{i=1}^n Az_i^t\|\leqslant \lambda$. We let $\varrho_\xi(y, 0, \sigma)=\|y\|$.     

\begin{lemma}Fix $\sigma>0$ and an operator $A:X\to Y$. \begin{enumerate}[(i)]\item If $N=N_\xi(\sigma)<\infty$, there exist $\ee_1, \ldots, \ee_{N+1}\in \rr$ with $\ee_i\leqslant \|A\|$ such that $\sum_{n=1}^{N+1}\ee_n>1/3\sigma$ and $s_{\xi,\ee_1}\ldots s_{\xi,\ee_{N+1}}(A^*B_{Y^*})\neq \varnothing$.   \item For every $n\in \nn\cup \{0\}$, the function $\varrho_\xi(\cdot, n, \sigma)$ is convex, radial, $1$-Lipschitz,  and if $\nn\ni n\leqslant N_\xi(\sigma)$, $\varrho_\xi(y, n, \sigma)\leqslant \|y\|+1$. Moreover, if $y\in Y$ is fixed, the sequence $(\varrho_\xi(y, n, \sigma))_{n=0}^\infty$ is non-decreasing.  \item For any $N\in \nn$, any $y\in Y$, any normally weakly null collection $(x_t)_{t\in \Omega_\xi}\subset \sigma B_X$, and any $\delta>0$, there exist $t\in \Omega_\xi$ and $x\in \text{\emph{co}}(x_s:s\preceq t)$ such that for every $0\leqslant n< N$, $\varrho_\xi(y+Ax, n, \sigma)<\delta +\varrho_\xi(y, n+1, u)$.  \end{enumerate}

\label{everything}
\end{lemma}

\begin{proof} 

$(i)$ Assume $N=N_\xi(\sigma)<\infty$. Note that this means $A$ is not the zero operator.  Fix $(x_t)_{t\in \Omega_{\xi, N+1}}\subset \sigma B_X$ normally weakly null such that for every $t\in MAX(\Omega_{\xi, n})$, $\|\sum_{i=1}^{N+1}Az_i^t\|>1$. For every $t\in MAX(\Omega_{\xi, N+1})$, fix $y^*_t\in B_{Y^*}$ such that $\text{Re\ }y^*_t(\sum_{i=1}^{N+1}Az_i^t)=\|\sum_{i=1}^{N+1}Az_i^t\|$.  Define $f:\Pi\Omega_{\xi, N+1}\to [-\sigma \|A\|,\sigma \|A\|]$ by $f(s,t)=\text{Re\ }A^*y^*_t(x_s)$.  By Theorem \ref{favorite theorem} and the remarks following it, there exist $a_1, \ldots, a_{N+1}\in \rr$ with $a_i\leqslant \sigma \|A\|$ and an extended pruning $(\theta, e):\Omega_{\xi, N+1}\to \Omega_{\xi, N+1}$ such that for every $(s,t)\in \Pi\Omega_{\xi, N+1}$, if $s\in \Lambda_{\xi, N+1, i}$, $f(\theta(s), e(t))\geqslant a_i$ and for each $t\in MAX(\Omega_{\xi, N+1})$, $\sum_{s\preceq t} \mathbb{P}_\xi(\iota_{\xi, N+1}(s))f(s, e(t)) \leqslant 1/3+\sum_{i=1}^{N+1}a_i$.  But for any $t\in MAX(\Omega_{\xi, N+1})$, $$1<\text{Re\ }y^*_t(\sum_{i=1}^{N+1} Az_i^t)=\sum_{s\preceq t}\mathbb{P}_\xi(\iota_{\xi, N+1}(s))f(s, e(t)) \leqslant 1/3+\sum_{i=1}^{N+1} a_i,$$ and $2/3\leqslant \sum_{i=1}^{N+1}a_i$.   Fix $\ee_1, \ldots, \ee_{N+1}\in \rr$ such that $\ee_i=a_i$ if $a_i\leqslant 0$, $0<\ee_i<a_i/\sigma$ if $a_i>0$, and $\sum_{i=1}^{N+1} \ee_i > 1/3\sigma$.  Then by Proposition \ref{characterize}$(iii)$, $s_{\xi, \ee_1}\ldots s_{\xi, \ee_{N+1}}(A^*B_{Y^*})\neq \varnothing$.  Indeed, $(x_{\theta(t)})_{t\in \Omega_{\xi, N+1}}\subset \sigma B_X$ is normally weakly null and for any $t\in MAX(\Omega_{\xi, N+1})$, $A^*y^*_{e(t)}\in A^*B_{Y^*}$ is such that for any $s\preceq t$, if $s\in \Lambda_{\xi, N+1, i}$, $\text{Re\ }A^*y^*_{e(t)}(x_{\theta(s)})=f(\theta(s), e(t))\geqslant a_i$.   

$(ii)$ By \emph{radial}, we mean that for any $y\in Y$ and any unimodular scalar $\ee$, $\varrho_\xi(y,n,\sigma)=\varrho_\xi(\ee y, n, \sigma)$. It is clear that $\varrho_\xi(\cdot, n, \sigma)$ is radial and $1$-Lipschitz.  Note that for any $n\in \nn$, any normally weakly null collection $(x_t)_{t\in \Omega_{\xi, n}}\subset \sigma B_X$ can be extended to a collection $(z_t)_{t\in \Omega_{\xi, n+1}}\subset \sigma B_X$ by letting the first $n$ levels of $(z_t)_{t\in \Omega_{\xi, n+1}}$ be the tree $(x_t)_{t\in \Omega_{\xi, n+1}}$ and by letting the $n+1^{st}$ level be filled with zeros.  This yields that $\varrho_\xi(\cdot, n, \sigma)\leqslant \varrho_\xi(\cdot, n+1, \sigma)$. We similarly deduce that $\varrho_\xi(\cdot, 0, \sigma)\leqslant \varrho_\xi(\cdot, 1, \sigma)$ by considering the collection $(x_t)_{t\in \Omega_{\xi, 1}}$ consisting of all zeros.   By definition of $N_\xi(\sigma)$, for any $n\leqslant N_\xi(\sigma)$ and any normally weakly null collection $(x_t)_{t\in \Omega_{\xi, n}}\subset \sigma B_X$, there exists $t\in MAX(\Omega_{\xi, n})$ such that $\|\sum_{i=1}^n Az_i^t\|\leqslant 1$, from which it follows that  $\varrho_\xi(y, n, \sigma)\leqslant \|y\|+1$ for any $y\in Y$ and $\nn\ni n\leqslant N_\xi(\sigma)$.  We show that $\varrho_\xi(\cdot, n, \sigma)$ is convex.   To obtain a contradiction, assume that for some $y_1, y_2\in Y$ and $0<\alpha<1$, with $y=\alpha y_1+(1-\alpha)y_2$,  $$\varrho_\xi(y, n, \sigma) > \delta+\alpha \varrho_\xi(y_1, n, \sigma)+(1-\alpha)\varrho_\xi(y_2, n, \sigma)$$ for some $\delta\in (0,1)$. Let $\varrho_1=\varrho_\xi(y_1, n, \sigma)$ and let $\varrho_2=\varrho_\xi(y_2, n, \sigma)$.  By the definition of $\varrho_\xi(y, n, \sigma)$, there exists a normally weakly null $(x_t)_{t\in \Omega_{\xi, n}}\subset \sigma B_X$ such that $$\inf_{t\in MAX(\Omega_{\xi, n})}\|y+\sum_{i=1}^n Az^t_i\|>\delta+\alpha \varrho_1+(1-\alpha)\varrho_2.$$   For each $t\in MAX(\Omega_{\xi, n})$, fix $y^*_t\in B_{Y^*}$ such that $y^*_t(y+\sum_{i=1}^n Az_i^t)=\|y+\sum_{i=1}^n Az_i^t\|$.   Define $f:\Pi\Omega_{\xi, n}\to \rr$ by $f(s, t)=\text{Re\ }y^*_t(y/n+Ax_s)$. Then \begin{align*} \inf_{t\in MAX(\Omega_{\xi, n})}\sum_{s\preceq t} \mathbb{P}_\xi(\iota_{\xi, n}(s))f(s,t) & = \inf_{t\in MAX(\Omega_{\xi, n})}\sum_{i=1}^n \sum_{t_{i-1}\prec s\preceq t_i}\mathbb{P}_\xi(\iota_{\xi, n}(s))\text{Re\ }y^*_t(y/n+Ax_s) \\ & =  \inf_{t\in MAX(\Omega_{\xi, n})}\|y+\sum_{i=1}^n Az^t_i\|  > \delta+\alpha \varrho_1+(1-\alpha)\varrho_2.\end{align*} Here, $\varnothing=t_0\prec \ldots \prec t_n=t$ are the initial segments of $t$ such that for each $1\leqslant i\leqslant n$, $t_i\in MAX(\Lambda_{\xi, n,i})$.    Then by Theorem \ref{favorite theorem} and the following remarks, there exist $a_1, \ldots, a_n\in \rr$ and a unit preserving extended pruning $(\theta, e):\Omega_{\xi, n}\to \Omega_{\xi, n}$ such that for every $(s,t)\in \Pi\Omega_{\xi, n}$, if $s\in \Lambda_{\xi, n, i}$, $f(\theta(s), e(t))\geqslant a_i$ and such that $$\delta+\alpha \varrho_1 + (1-\alpha)\varrho_2< \sum_{i=1}^n a_i.$$  

Next, let $\mathcal{P}$ be a finite partition of $[-r_1, r_1]\times [-r_2, r_2]$ into sets of diameter (with respect to the $\ell_1^2$ metric) less than $\delta/2n$, where $r_1=\|y_1\|/n+\sigma\|A\|$ and $r_2=\|y_2\|/n+\sigma\|A\|$.   Define the function $g:\Pi\Omega_{\xi, n}\to \mathcal{P}$ by letting $g(s,t)$ be the member $S$ of $\mathcal{P}$ such that $$(\text{Re\ }y^*_{e(t)}(y_1/n+Ax_{\theta(s)}), \text{Re\ }y^*_{e(t)}(y_2/n+ Ax_{\theta(s)}))\in S.$$  By applying Lemma \ref{combinatorial}, we may fix an $n$-tuple $(S_i)_{i=1}^n\in \mathcal{P}^n$ and a unit preserving extended pruning $(\theta', e'):\Omega_{\xi, n}\to \Omega_{\xi, n}$ such that for every $(s,t)\in \Pi\Omega_{\xi, n}$, if $s\in \Lambda_{\xi, n,i}$, $$(\text{Re\ }y^*_{e\circ e'(t)}(y_1/n+Ax_{\theta\circ \theta'(s)}), \text{Re\ }y^*_{e\circ e'(t)}(y_2/n+ Ax_{\theta\circ \theta'(s)}))\in S_i.$$ For each $1\leqslant i\leqslant n$, fix $(b_i, c_i)\in S_i$.   For each $s\in \Omega_{\xi, n}$, let $u_s=x_{\theta\circ \theta'(s)}$.  For each $t\in MAX(\Omega_{\xi, n})$, let $u^*_t=y^*_{e\circ e'(t)}$.   Fix $1\leqslant i\leqslant n$ and $(s,t)\in \Pi\Omega_{\xi, n}$ with $s\in \Lambda_{\xi, n, i}$.  Then $$a_i\leqslant \text{Re\ }u^*_t(y/n+Au_s)= \alpha \text{Re\ }u^*_t(y_1/n+Au_s)+(1-\alpha)\text{Re\ }u^*_t(y_2/n+Au_s) < \alpha b_i+(1-\alpha) c_i+\delta/2n,$$ and  $$ \delta+\alpha \varrho_1+(1-\alpha)\varrho_2 \leqslant \sum_{i=1}^n a_i < \delta/2+\alpha\sum_{i=1}^n b_i+(1-\alpha)\sum_{i=1}^n c_i.$$ From this it follows that either $\delta/2+\varrho_1< \sum_{i=1}^n b_i$ or $\delta/2+\varrho_2<\sum_{i=1}^n c_i$.   We assume $\delta/2+\varrho_1<\sum_{i=1}^n  b_i$, with the other case being identical.  Let $\varrho=\sum_{i=1}^n b_i- (\delta/2+\varrho_1)>0$.   Then since $(u_t)_{t\in \Omega_{\xi, n}}\subset \sigma B_X$ is normally weakly null, we deduce that \begin{align*} \varrho_1 & \geqslant \inf_{t\in MAX(\Omega_{\xi, n})} \|y+\sum_{i=1}^n Az_i^t\| \geqslant \inf_{t\in MAX(\Omega_{\xi, n})} \text{Re\ }u^*_t(y+\sum_{i=1}^n Az_i^t) \\ & = \inf_{t\in MAX(\Omega_{\xi, n})} \sum_{i=1}^n\sum_{t_{i-1}\prec s\preceq t_i} \mathbb{P}_\xi(\iota_{\xi, n}(s)) \text{Re\ }u^*_t(y_1/n+Au_s) \\ &  \geqslant -n(\delta/2n)+ \sum_{i=1}^n  b_i = \varrho_1+\varrho+\delta/2-\delta/2>\varrho_1.\end{align*} This contradiction finishes $(ii)$.

$(iii)$  For $u\in MAX(\Lambda_{\xi, N,1})$, we let $B_u=\{s\in \Omega_{\xi, N}: u\prec s\}$.  To obtain a contradiction, assume $(iii)$ fails.  Then there exist $y\in Y$, $\delta>0$, a normally weakly null collection $(x_t)_{t\in \Omega_{\xi, N}}\subset \sigma B_X$, natural numbers $(n_u)_{u\in MAX(\Lambda_{\xi, N, 1})}\subset \{0, \ldots N-1\}$, and functionals $(y^*_t)_{t\in MAX(\Omega_{\xi, N})}\subset B_{Y^*}$ such that \begin{enumerate}[(i)]\item for every $u\in MAX(\Lambda_{\xi, N, 1})$, $z^u:=\sum_{s\preceq u}\mathbb{P}_\xi(\iota_{\xi, N}(s))x_s$ is such that $\varrho_\xi(y+Az^u, n_u, \sigma)>\delta+\varrho_\xi(y, n_u+1, \sigma)$, \item for each $u\in MAX(\Lambda_{\xi, N, 1})$, $x_s=0$ for every $s\in B_u\cap (\cup_{i=n_u+2}^N \Lambda_{\xi, N, i})$,  \item for every $t\in MAX(\Omega_{\xi, N})$, $$y^*_t(y+\sum_{i=1}^N Az^t_i)=y^*_t(y+\sum_{i=1}^{n_{t_1}+1} Az^t_i)=\|y+\sum_{i=1}^{n_{t_1}+1} Az_i^t\|>\delta+\varrho_\xi(y, n_{t_1}+1, \sigma).$$ \end{enumerate}  To see why these collections exist, the vectors $(x_t)_{t\in \Lambda_{\xi, N, 1}}\subset \sigma B_X $ and $n_u\in \{0, \ldots N-1\}$ are chosen so that for every $u\in MAX(\Lambda_{\xi, N, 1})$, $\varrho_\xi(y+Az^u, n_u, \sigma)>\delta +\varrho_\xi(y, n_u+1, \sigma)$, as we may do, since $\Lambda_{\xi, N, 1}$ is identifiable with $\Omega_\xi$.    Then for each $u\in MAX(\Lambda_{\xi, N, 1})$, the vectors $(x_s)_{s\in B_u\cap (\cup_{i=2}^{n_u+1} \Lambda_{\xi, N, i})}\subset \sigma B_X$ are chosen to witness that $\varrho_\xi(y+Az^u, n_u, \sigma)>\delta +\varrho_\xi(y, n_u+1, \sigma)$, and $(x_s)_{s\in B_u\cap (\cup_{i=n_u+2}^N\Lambda_{\xi, N, i})}$ are chosen to be zero.  We then choose for each $t\in MAX(\Omega_{\xi, N})$ some $y^*_t\in B_{Y^*}$ to norm $y+\sum_{i=1}^N Az_i^t=y+\sum_{i=1}^{n_{t_1}+1} Az_i^t$.  

Define $f:\Pi\Omega_{\xi, N}\to \rr$ by \begin{displaymath}
   f(s,t) = \left\{
     \begin{array}{lr}
       \text{Re\ }y^*_t\Bigl(\frac{y}{n_{t_1}+1}+Ax_s\Bigr) - \frac{\varrho_\xi(y, n_{t_1}+1, \sigma)}{n_{t_1}+1} & : s\in \cup_{i=1}^{n_{t_1}+1} \Lambda_{\xi, N,i}\\
       0 & : s\in \cup_{i=n_{t_1+2}}^N \Lambda_{\xi, N, i}
     \end{array}
   \right.
\end{displaymath}

Note that for any $t\in MAX(\Omega_{\xi, N})$, $$\sum_{s\preceq t}\mathbb{P}_\xi(\iota_{\xi, N}(s))f(s,t)= y^*_t(y+\sum_{i=1}^{n_{t_1}+1} Az_i^t)-\varrho_\xi(y, n_{t_1}+1, \sigma)>\delta.$$  From this, Theorem \ref{favorite theorem}, and the remarks following it, there exist $a_1, \ldots, a_N$ and a unit preserving extended pruning $(\theta', e'):\Omega_{\xi, N}\to \Omega_{\xi, N}$ such that $\delta/2<\sum_{i=1}^N a_i$ and such that for every $(s,t)\in \Pi\Omega_{\xi, N}$ with $s\in \Lambda_{\xi, N, i}$, $f(\theta'(s), e'(t))\geqslant a_i$.   Note that for each $u\in MAX(\Lambda_{\xi, N, 1})$, there exists a unique $e_0(u)\in MAX(\Lambda_{\xi, N, 1})$ such that $\theta'(B_u)\subset B_{e_0(u)}$.  Indeed, $B_u\cap \Lambda_{\xi, N, 2}$ is a single unit, which must be mapped into a single unit in $\Lambda_{\xi, N, 2}$, and therefore into $B_{e_0(u)}$ for a unique $e_0(u)\in MAX(\Lambda_{\xi, N,1})$.  Define $m:MAX(\Lambda_{\xi, N,1})\to \{0, \ldots, N-1\}$ by letting $m(u)=n_{e_0(u)}$.  Then by Proposition \ref{recall}, there exists an extended pruning $(\theta'', e''):\Lambda_{\xi, N, 1}\to \Lambda_{\xi, N,1}$ and $n\in \{0, \ldots, N-1\}$ such that for every $u\in MAX(\Lambda_{\xi, N, 1})$, $n=m(e''(u))=n_{e_0\circ e''(u)}$.   Define $\theta:\Omega_{\xi,N}\to \Omega_{\xi, N}$ and $e:MAX(\Omega_{\xi, N})\to \Omega_{\xi, N}$ by letting $$\theta|_{\Lambda_{\xi, N, 1}}= \theta'\circ \theta'',$$ for $u\in MAX(\Lambda_{\xi, N, 1})$, let $j_u:B_u\to B_{e_0\circ e''(u)}$ be the canonical identification,  $$\theta|_{B_u}=\theta'\circ j_u,$$ $$e|_{MAX(B_u)}=e'\circ j_u.$$ 

Note that for any $(s, t)\in \Pi\Omega_{\xi, N}$ with $s\in \Lambda_{\xi, N, i}$, $f(\theta(s), e(t))\geqslant a_i$.  

We also note that $\sum_{i=1}^{n+1} a_i\geqslant \sum_{i=1}^N a_i>\delta/2$.  If $n+1=N$, this is clear.  Otherwise fix $n+1<i\leqslant N$ and $(s,t)\in \Pi\Omega_{\xi, N}$ such that $s\in \Lambda_{\xi, N, i}$.  Fix $u\in MAX(\Lambda_{\xi, N, 1})$ such that $u\preceq s$ and note that, since $n+1=n_{e_0\circ e''(u)}+1<i$, $$a_i\leqslant f(\theta(s), e(t))=0$$ by the definition of $f$.  Thus $a_i\leqslant 0$ for each $n+1<i\leqslant N$, and $\sum_{i=1}^{n+1} a_i\geqslant \sum_{i=1}^N a_i>\delta/2$.

Let $\kappa:\Omega_{\xi, n+1}\to \Omega_{\xi, N}$ be the canonical identification of $\Omega_{\xi, n+1}$ with the first $n+1$ levels of $\Omega_{\xi, N}$ and let $\kappa':MAX(\Omega_{\xi, n+1})\to MAX(\Omega_{\xi, N})$ be any function such that $(\kappa, \kappa')$ is an extended pruning.     For each $s\in \Omega_{\xi, n+1}$, let $v_s= x_{\theta(\kappa(s))}$ and for each $t\in MAX(\Omega_{\xi, n+1})$, let $v^*_t=y^*_{e(\kappa'(t))}$.    By our remark above, for any $1\leqslant i\leqslant n+1$ and each $(s,t)\in \Omega_{\xi, n+1}$ such that $s\in \Lambda_{\xi, n+1, i}$, $$\text{Re\ }v^*_t(\frac{y}{n+1} +Av_s) -\frac{\varrho_\xi(y, n+1, \sigma)}{n+1} = f(\theta(\kappa(s)), e(\kappa'(t)))\geqslant a_i.$$  From this it and the definition of $\varrho_\xi(y, n+1, \sigma)$, it follows that \begin{align*} 0 & \geqslant \inf_{t\in MAX(\Omega_{\xi, n+1})} \|y+\sum_{i=1}^{n+1}\sum_{t_{i-1}\prec s\preceq t_i} \mathbb{P}_\xi(\iota_{\xi, n+1}(s))Av_s\| -\varrho_\xi(y, n+1, \sigma) \\ & \geqslant \inf_{t\in MAX(\Omega_{\xi, n+1})} \text{Re\ }v^*_t\Bigl(y+\sum_{i=1}^{n+1}\sum_{t_{i-1}\prec s\preceq t_i} \mathbb{P}_\xi(\iota_{\xi, n+1}(s))Av_s\Bigr) -\varrho_\xi(y, n+1, \sigma) \\ & =\inf_{t\in MAX(\Omega_{\xi, n+1})} \sum_{i=1}^{n+1}\sum_{t_{i-1}\prec s\preceq t_i}\mathbb{P}_\xi(\iota_{\xi, n+1}(s))\Bigl[\text{Re\ }v^*_t\bigl(\frac{y}{n+1}+Av_s\bigr)- \frac{\varrho_\xi(y, n+1,\sigma)}{n+1}\Bigr] \\ & \geqslant \sum_{i=1}^{n+1}\sum_{t_{i-1}\prec s\preceq t_i}\mathbb{P}_\xi(\iota_{\xi, n+1}(s))a_i=\sum_{i=1}^{n+1}a_i>\delta/2.\end{align*} This contradiction finishes $(iii)$.

\end{proof}

The next lemma is a modification of a result of Lancien \cite{Lancien}, which was proved in the case $K=B_{X^*}$.   

\begin{lemma}Let $X$ be a Banach space, $K,L,M\subset X^*$ $w^*$-compact with $K$ convex.    \begin{enumerate}[(i)]\item For any $\ee>0$ and any ordinal $\xi$, $L+s_\ee^\xi(M)\subset s_\ee^\xi(L+M)$. \item For any ordinal $\xi$ and for any $n\in \nn$, if $Sz(K, \ee)>\xi$, $Sz(K, \ee/n)>\xi n$.   \item If $Sz(K)=\omega^{\xi+1}$, $\textbf{\emph{p}}_\xi(K)\geqslant 1$.  \end{enumerate}

\label{lemma1}
\end{lemma}

\begin{proof}$(i)$ This is a standard inductive proof. 

$(ii)$ For $\ee>0$ and $n\in \nn$, applying $(i)$ successively yields that \begin{align*} s^{\xi n}_{\ee/n}(K) & = s^{\xi n}_\ee(nK)\supset s^{\xi n}_\ee (K+\ldots +K)\supset s^{\xi(n-1)}_\ee(K+\ldots +K+s^\xi_\ee(K)) \\ & \supset s^{\xi(n-2)}_\ee(K+\ldots +K+s^\xi_\ee(K)+s^\xi_\ee(K)) \supset \ldots \supset s^\xi_\ee(K)+\ldots +s^\xi_\ee(K) .\end{align*}

  If $Sz(K, \ee)>\xi$, this last sum is non-empty, and so is $s^{\xi n}_{\ee/n}(K)$, whence $Sz(K, \ee/n)>\xi n$.

$(iii)$ If $Sz(K)=\omega^{\xi+1}$, there exists $\ee_1>0$ such that $Sz(K, \ee_1)>1$.  Then for any $n\in \nn$, $Sz(K, \ee_1/n)>n$, whence we easily deduce $(iii)$.

\end{proof}

We recall the following from \cite{CD}. 

\begin{proposition} Fix an operator $A:X\to Y$ and $0<\sigma, \tau\leqslant 1$, and an ordinal $\xi$. \begin{enumerate}[(i)]\item If $\rho_\xi(\sigma;A)\leqslant \sigma\tau$, then $\delta^{w^*}_\xi(6\tau;A^*)\geqslant \sigma\tau$.  \item If $\delta^{w^*}_\xi(\tau;A^*) \geqslant \sigma\tau$, then $\rho_\xi(\sigma;A)\leqslant \sigma\tau$.  \end{enumerate}
\label{duality}
\end{proposition}

\begin{theorem} Let $A:X\to Y$ be an operator with norm not exceeding $1$ and fix $0<\sigma\leqslant 1$.   \begin{enumerate}[(i)]\item If $N=N_\xi(\sigma)<\infty$, there exists a $2$-equivalent norm $|\cdot|$ on $Y$ such that $\rho_\xi(\sigma;A:X\to (Y, |\cdot|))\leqslant 1/N$.  \item If $N_\xi(\sigma)=\infty$, then for any $N\in \nn$, there exists a $2$-equivalent norm $|\cdot|$ on $Y$ such that $\rho_\xi(\sigma;A:X\to (Y, |\cdot|))\leqslant 1/N$.  \end{enumerate}

\label{renorming1}
\end{theorem}

\begin{proof} We prove $(i)$ and $(ii)$ simultaneously.  In the case of $(i)$, let $N=N_\xi(\sigma)$ and in the case of $(ii)$, fix $N\in \nn$ arbitrary.  Let $g(y)=N^{-1}\sum_{n=0}^{N-1}\varrho_\xi(y, n, \sigma)$.   Let $B=\{y\in Y: g(y)\leqslant 2\}$.  It follows from Lemma \ref{everything} that $B$ is the unit ball of an equivalent norm $|\cdot|$ on $Y$.  Furthermore, still using Lemma \ref{everything}, for any normally weakly null $(x_t)_{t\in \Omega_\xi}\subset B_X$ and any $y\in B$, \begin{align*} \inf_{t\in \Omega_\xi, x\in \text{co}(x_s:s\preceq t)} g(y+\sigma Ax)  & = \inf_{t\in \Omega_\xi, x\in \text{co}(x_s:s\preceq t)} N^{-1}\sum_{n=0}^{N-1}\varrho_\xi(y+\sigma Ax, n, \sigma) \\ & \leqslant N^{-1}\sum_{n=1}^N \varrho_\xi(y, n, \sigma) \leqslant g(y)+ \frac{\varrho_\xi(y, N, \sigma)-\|y\|}{N} \\ & \leqslant 2 + \frac{1}{N}.\end{align*} Since $g$ is convex and $g(0)<1$, it follows that $$\inf_{t\in \Omega_\xi, x\in \text{co}(x_s:s\preceq t)} g\Bigl(\frac{N}{N+1}(y+\sigma Ax)\Bigr)<2,$$ whence $$\inf_{t\in \Omega_\xi, x\in \text{co}(x_s:s\preceq t)} |y+\sigma Ax| \leqslant 1+1/N.$$  

Now, for a general $B$-tree $B$ with $o(B)=\omega^\xi$ and a weakly null collection $(x_t)_{t\in B}\subset \sigma B_X$, by \cite[Proposition $2.1$]{CD}, there exists a $B$-tree $B_{\omega^\xi}$ with $o(B_{\omega^\xi})=\omega^\xi$ and a monotone map $\theta:B_{\omega^\xi}D\to B$ such that $(x_{\theta(t)})_{t\in B_{\omega^\xi}D}$ is normally weakly null. Here, $D$ is our fixed weak neighborhood basis at $0$ in $X$. Since $o(\Gamma_\xi)=\omega^\xi$, there exists a monotone $\phi:\Gamma_\xi\to B_{\omega^\xi}$ which preserves lengths.  Now we define $\phi_0:\Omega_\xi\to B_{\omega^\xi}D$ by letting $\phi_0((\zeta_i, U_i)_{i=1}^n)=(\mu_i, U_i)_{i=1}^n$, where $(\mu_i)_{i=1}^n=\phi((\zeta_i)_{i=1}^n)$.  Then the collection $(x_{\theta\circ \phi_0(t)})_{t\in \Omega_\xi}\subset \sigma B_X$ is normally weakly null.  Thus by the previous argument, for any $\delta>0$, there exists $t\in \Omega_\xi$ and $$x\in \text{co}(x_{\theta\circ \phi_0(s)}:s\preceq t)\subset \text{co}(x_s: s\preceq \theta\circ \phi_0(t))$$ such that $|x|\leqslant 1+1/N$.   

\end{proof}

\begin{theorem} Fix $p,q,r$ such that $1/p+1/q=1$, $1\leqslant p<\infty$, $1< r<q$.  Then for any ordinal $\xi$ and any operator $A:X\to Y$ with $\|A\|\leqslant 1$ and $\sup_{0<\ee\leqslant 1} \ee^p Sz_\xi(A, \ee)=\alpha\leqslant\infty$, there exists a constant $\beta=\beta(\alpha, p,r)$ such that for any $n\in \nn$, any positive scalars $(a_i)_{i=1}^n$, and any normally weakly null $(x_t)_{t\in \Omega_{\xi, n}}\subset B_X$, $$\inf_{t\in MAX(\Omega_{\xi,n})} \|\sum_{i=1}^n a_i A z^t_i\|\leqslant \beta\bigl(\sum_{i=1}^n a_i^r\bigr)^{1/r}.$$   

\label{cronenberg}
\end{theorem}

\begin{proof} Fix $n\in\nn$, a normally weakly null $(x_t)_{t\in \Omega_{\xi,n}}\subset B_X$, and positive scalars $(a_i)_{i=1}^n$. By homogeneity, we may assume $\max_{1\leqslant i\leqslant n}a_i=1$.  Fix $$\lambda<\inf_{t\in MAX(\Omega_{\xi,n})} \|\sum_{i=1}^n a_i Az_i^t\|$$ and fix $\delta>0$ such that $$\lambda+(n+1)\delta < \inf_{t\in MAX(\Omega_{\xi,n})} \|\sum_{i=1}^n a_i A z_i^t\|.$$   For every $t\in MAX(\Omega_{\xi,n})$, fix $y^*_t\in B_{Y^*}$ such that $\text{Re\ }A^*y^*_t(\sum_{i=1}^n a_i z^t_i)=\|\sum_{i=1}^n a_i A z_i^t\|$.  Let $f:\Pi \Omega_{\xi,n}\to [-1,1]$ be given by $f(s,t)=a_i\text{Re\ }A^*y^*(x_s)$, where $s\in \Lambda_{\xi,n,i}$.  Then $$\lambda+(n+1)\delta< \inf_{t\in MAX(\Omega_{\xi,n})} \|\sum_{i=1}^n a_i A z_i^t\|= \inf_{t\in MAX(\Omega_{\xi,n})} \text{Re\ }\sum_{i=1}^n A^*y^*_t(z^t_i)=\inf_{t\in MAX(\Omega_{\xi,n})} \sum_{s\preceq t}f(s,t).$$  By Theorem \ref{favorite theorem} and the remarks following it, we may fix $b_1, \ldots, b_n\in [-1,1]$ and a unit preserving extended pruning $(\theta, e):\Omega_{\xi,n}\to \Omega_{\xi,n}$ such that $\lambda+n\delta<\sum_{i=1}^n b_i$ and for each $1\leqslant i\leqslant n$, $(s,t)\in \Pi \Omega_{\xi,n}$ with $s\in \Lambda_{\xi,n,i}$, $f(\theta(s), e(t))\geqslant b_i-\delta$.   By applying Theorem \ref{combinatorial}, we may fix a unit preserving extended pruning $(\theta', e'):\Omega_{\xi, n}\to \Omega_{\xi,n}$ and numbers $d_i$ such that for every $1\leqslant i\leqslant n$ and $(s,t)\in \Pi\Omega_{\xi,n}$ with $s\in \Lambda_{\xi,n,i}$, $d_i-\delta\leqslant \text{Re\ }A^*y^*_{e\circ e'(t)}(x_{\theta\circ \theta'(s)})\leqslant d_i$.  Note that $b_i-\delta\leqslant a_id_i$ for each $1\leqslant i\leqslant n$, so that $\lambda<\sum_{i=1}^n a_i d_i$.

For each $j\in \nn$, let $$B_j=\{i\leqslant n: d_i-\delta\in (2^{-\frac{j}{p}}, 2^{-\frac{j-1}{p}}]\}.$$  We claim that $|B_j|\leqslant \alpha 2^{-j} $.    Indeed, for $1\leqslant i\leqslant n$, let $\ee_i=2^{-\frac{j}{p}}$ if $i\in B_j$, and let $\ee_i=0$ otherwise.  By Proposition \ref{characterize} applied with $\sigma=1$, and $d_i-\delta>\ee_i$ for those $i\in B_j$, we deduce that $\varnothing\neq s_{\xi, \ee_1}\ldots s_{\xi, \ee_n}(A^*B_{Y^*})=s^{\xi |B_j|}_{2^{-\frac{j}{p}}}(A^*B_{Y^*})$, so that $Sz_\xi(A, 2^{-\frac{j}{p}})>|B_j|>\alpha 2^j$.  However, $$Sz_\xi(A, 2^{-\frac{j}{p}})\leqslant \alpha (2^{-\frac{j}{p}})^{-p} = \alpha 2^j,$$  a contradiction.   Thus $|B_j|\leqslant \alpha 2^j$.  Note also that $\{i\leqslant n: \delta_i-\delta>0\}=\cup_j B_j$.

Let $s$ be such that $1/r+1/s=1$. Let $N_0=\{\ee e_i^*:|\ee|=1, i\in \nn\}\subset c_{00}$ and $$N_{k+1}=N_k\cup \{2^{-1/s}(f+g): f,g\in N_k, \max\supp(f)<\min \supp(g)\}\subset c_{00},$$ $N=\cup_{k=0}^\infty N_k$. Define the norm $\|\cdot\|_T$ on $c_{00}$ by $\|x\|_T=\sup_{f\in N}|f(x)|$.  An easy proof by induction yields that $N\subset B_{\ell_s}$, so that $\|\cdot\|_T\leqslant \|\cdot\|_{\ell_r}$. Actually, it was shown in \cite{AT} that this space $T$ is isomorphic to $\ell_r$, but we do not need this.    Let $m=\lceil \alpha\rceil$ and note that for any subset $B$ of $\nn$ with $|B|\leqslant m 2^j$, $\|\sum_{i\in B} e_i^*\|_{T^*} \leqslant m 2^{\frac{j}{s}}$. From this we deduce that, with $f=\sum_j 2^{-\frac{j-1}{p}}\sum_{i\in B_j}e_i^*$, $$\|f\|_{T^*}\leqslant 2^\frac{1}{p}m\sum_{j=1}^\infty (2^{\frac{1}{s}-\frac{1}{p}})^j= \frac{m 2^{\frac{1}{p}+\frac{1}{s}}}{2^{\frac{1}{p}}-2^{\frac{1}{s}}}.$$  Therefore \begin{align*} \lambda & <\sum_{i=1}^n a_i d_i \leqslant \delta n+\sum_{i=1}^n a_i (d_i-\delta) \leqslant \delta n +\sum_j \sum_{i\in B_j} a_i(d_i-\delta) \\ & \leqslant \delta n +\sum_j \sum_{i\in B_j} a_i 2^{- \frac{j-1}{p}} =  \delta n+f(\sum_{i=1}^n a_ie_i)\\ & \leqslant \delta n + \frac{m 2^{\frac{1}{p}+\frac{1}{s}}}{2^{\frac{1}{p}}-2^{\frac{1}{s}}}\|\sum_{i=1}^n a_i e_i\|_T \leqslant \delta n+\frac{m 2^{\frac{1}{p}+\frac{1}{s}}}{2^{\frac{1}{p}}-2^{\frac{1}{s}}}\bigl(\sum_{i=1}^n a_i^r\bigr)^{1/r}.  \end{align*}

Since $\delta>0$ and $\lambda<\inf_{t\in MAX(\Omega_{\xi,n})} \|\sum_{i=1}^n a_i Az_i^t\|$ were arbitrary, we deduce that the conclusion is satisfied with $$\beta(\alpha,p,r) = \frac{\lceil \alpha\rceil 2^{\frac{1}{p}+\frac{1}{s}}}{2^{\frac{1}{p}}-2^{\frac{1}{s}}}.$$

\end{proof}

\begin{corollary} For any $p,q,r$ such that $1\leqslant p<\infty$, $1/p+1/q=1$, and $1\leqslant r<q$, for any constants $c, C>0$, there exists a number $\gamma=\gamma(c,C,p,r)$ such that if $A:X\to Y$ is an operator and $\xi$ is an ordinal such that $\|A\|\leqslant c$ and $\sup_{0<\ee< \|A\|}\ee^p Sz_\xi(A, \ee)\leqslant C$, then for any $0<\sigma<\|A\|$ there exists a $2$-equivalent norm $|\cdot|$ on $Y$ such that $\rho_\xi(\sigma;A:X\to (Y, |\cdot|))\leqslant \gamma \sigma^r$. 

\label{things}
\end{corollary}

\begin{proof} Let $B= c^{-1}A$, so that $\|B\|\leqslant 1$. Note that $$\sup_{0<\ee<1} \ee^p Sz_\xi(B, \ee) = \sup_{0<\ee<1}\ee^p Sz_\xi(A, c\ee) = c^{-p}\sup_{0<\ee<c} \ee^p Sz_\xi(A, \ee)\leqslant c^{-p} C.$$   Then if $\beta=\beta(c^{-p}C, p, r)$ is the number from Theorem \ref{cronenberg}, for any $n\in \nn$, any positive scalars $(a_i)_{i=1}^n$, and any normally weakly null $(x_t)_{t\in \Omega_{\xi,n}}\subset B_X$, $$\inf_{t\in MAX(\Omega_{\xi,n})} \|\sum_{i=1}^n a_i B z_i^t\|\leqslant \beta\bigl(\sum_{i=1}^n a_i^r\bigr)^{1/r}.$$  From this it easily follows that for any $0<\sigma<1$, $N_\xi(\sigma;B) \geqslant 2^{-1}(\beta \sigma)^{-r}$. Indeed, if $(x_t)_{t\in \Omega_{\xi, n}}\subset \sigma B_X$ is normally weakly null and $1 <\|\sum_{i=1}^{n+1} B z_i^t\|$ for all $t\in MAX(\Omega_{\xi,n})$, then set $x'_t=\sigma^{-1} x_t$. Since $\|B\|\leqslant 1$, $n>0$.   By passing to a subtree and relabeling, we may assume $(x'_t)_{t\in \Omega_{\xi,n+1}}\subset B_X$ is itself normally weakly null, whence $$1\leqslant \inf_{t\in MAX(\Omega_{\xi,n})} \|\sum_{i=1}^{n+1} \sigma\sum_{\Lambda_{\xi,n,i}\ni s\preceq t} \mathbb{P}_\xi(\iota_{\xi,n}(s))Bx_s\| \leqslant \beta\bigl(\sum_{i=1}^{n+1} \sigma^r\bigr)^{1/r}= \beta\sigma (n+1)^{1/r}.$$   Then $$\frac{1}{\beta^r\sigma^r} \leqslant n+1 \leqslant 2n.$$  From this we deduce the existence of a $2$-equivalent norm $|\cdot|$ on $Y$ such that $\rho_\xi(\sigma;B:X\to (Y, |\cdot|))\leqslant 1/N_\xi(\sigma;B)\leqslant 2\beta^r \sigma^r$.  Then for $0<\sigma_1<\|A\|$, let $\sigma:=c^{-1}\sigma_1\in (0,1)$. Then by Theorem \ref{renorming1}, exists a $2$-equivalent norm $|\cdot|$ on $Y$ such that \begin{align*} \rho_\xi(\sigma_1;A:X\to (Y, |\cdot|)) & = \rho_\xi(c\sigma_1; B:X\to (Y, |\cdot|)) \leqslant 2\beta^r c^r\sigma_1^r.\end{align*}

Here we have used the obvious fact that for any $\sigma>0$, $\rho_\xi(\sigma_1;A:X\to (Y, |\cdot|))=\rho_\xi(c\sigma_1; c^{-1}A:X\to (Y, |\cdot|))$.

\end{proof}

\section{Proof of the Main Theorem}

\begin{lemma} Let $\xi$ be an ordinal and $A:X\to Y$ an operator such that $\delta=\delta^{w^*}_\xi(\ee;A^*:Y^*\to X^*)>0$. Then for any $0<r\leqslant 1$ and any $w^*$-compact $K\subset r B_{Y^*}$, $s_{2\ee}^{\omega^\xi}(A^*K)\subset (r-\delta)A^*B_{Y^*}$ if $r-\delta\geqslant 0$, and $s_{2\ee}^{\omega^\xi}(A^*K)=\varnothing$ if $r-\delta< 0$. 
\label{stranger}

\end{lemma}

\begin{proof} Fix $x^*\in s_{2\ee}^{\omega^\xi}(A^*K)$.  By \cite[Lemma $4.2$]{CD}, there exists a tree $T$ with $o(T)=\omega^\xi+1$ and a $w^*$-closed, $(A^*, \ee)$-separated collection $(y_t^*)_{t\in T}\subset K$ such that $A^*y^*_\varnothing= x^*$. Let $B=T\setminus\{\varnothing\}$ and let $z_t^*=\ee^{-1}(y^*_t-y^*_{t^-})$ for $t\in B$.  Then $(z^*_t)_{t\in B}$ is $w^*$-null and $(A^*, 1)$-large.  

First suppose that $r-\delta\geqslant 0$.  If $y^*_\varnothing=0$,  $x^*=0\in (r-\delta)A^*B_{Y^*}$ as desired. Assume $\|y^*_\varnothing\|\neq 0$.  Then by homogeneity together with the fact that $\|y^*_\varnothing\|\leqslant 1$, so $(\|y^*_\varnothing\|^{-1}z_t^*)_{t\in B}$ is still $(A^*, 1)$-large,  \begin{align*} (1+\delta)\|y^*_\varnothing\| & \leqslant  \sup_{t\in B}\|y^*_\varnothing +\ee \|y^*_\varnothing\|\sum_{s\preceq t} z^*_s\| \\ & \leqslant \sup_{t\in B}\Bigl\|\bigl(1-\|y^*_\varnothing\|)y^*_\varnothing+ \|y^*_\varnothing\|(y^*_\varnothing+\sum_{s\preceq t}\ee z_s)\Bigr\| \\ & \leqslant \sup_{t\in B} (1-\|y^*_\varnothing\|)\|y^*_\varnothing\|+ \|y^*_\varnothing\| \|y^*_t\| \leqslant (1-\|y^*_\varnothing\|)\|y^*_\varnothing\|+ \|y^*_\varnothing\| r.\end{align*} Rearranging yields $\|y^*_\varnothing\|\leqslant r-\delta$, finishing the $r-\delta\geqslant 0$ case.

Now suppose that $r-\delta< 0$. If $y^*_\varnothing\neq 0$, we may run the same proof as in the previous paragraph to arrive at the contradiction $\|y^*_\varnothing\|\leqslant r-\delta<0$.  If $y^*_\varnothing$, fix any $y^*\in S_{Y^*}$.  Then $$1+\delta  \leqslant \sup_{t\in B}\|y^*+\ee\sum_{s\preceq t}z^*_s\| = \sup_{t\in B}\|y^*+y^*_t\| \leqslant 1+r,$$   a contradiction.

\end{proof}

\begin{proof}[Proof of Theorem \ref{main theorem}]

First suppose that for some $1\leqslant p<\infty$ and all $0<\ee<1$, $\delta^{w^*}_\xi(\ee;A^*:Y^*\to X^*)\geqslant c \delta^p$.   Then by Lemma \ref{stranger}, for every $0<\ee<1$, if $\delta=\delta^{w^*}_\xi(\ee;A^*:Y^*\to X^*)\geqslant c\ee^p$, $s_{2\ee}^{\omega^\xi n}(A^*B_{Y^*})\subset (1-\delta n)A^*B_{Y^*}$ for any $n\in \nn$ such that $1-\delta n \geqslant 0$. Therefore $Sz_\xi(A^*B_{Y^*}, 2\ee) \leqslant 1+1/\delta \leqslant 1+1/c\ee^p$.  From this we easily deduce that $\textbf{p}_\xi(A)\leqslant p<\infty$.   

   Since $\textbf{p}_\xi(A)$ is invariant under renorming $Y$, if there exists an equivalent norm $|\cdot|$ on $Y$ such that $A^*:(Y^*, |\cdot|)\to X^*$ is $w^*$-$\xi$-AUC with power type (resp. with power type $p$), $\textbf{p}_\xi(A)<\infty$ (resp. $\textbf{p}_\xi(A)\leqslant p$).

Now suppose $1\leqslant \textbf{p}_\xi(A)<\infty$ and fix $\textbf{p}_\xi(A)<p_1<\infty$.  Let $p=\frac{2\textbf{p}_\xi(A) +p_1}{3}$ and let $s=\frac{\textbf{p}_\xi(A)+2p_1}{3}$.  Let $r,q$ be such that $1/p+1/q=1/r+1/s=1$.   By Corollary \ref{things}, there exists a constant $\gamma\geqslant 1$ such that for every $n\in\nn$, there exists a $2$-equivalent norm $|\cdot|_n$ on $Y$ such that $\rho_\xi((2^{-n}/6\gamma)^{\frac{1}{r-1}};A:X\to (Y, |\cdot|_n)) \leqslant \gamma (2^{-n}/6\gamma)^s$.  Here, we use $|\cdot|_n$ to denote the equivalent norm as well as the corresponding dual norm. Let $\tau=2^{-n}/6$ and let $\sigma=(\tau/\gamma)^\frac{1}{r-1}$. Note that $\gamma \sigma^r =\sigma\tau$, so that $\rho_\xi((\tau/\gamma)^\frac{1}{r-1};A:X\to (Y, |\cdot|_n))\leqslant \gamma \sigma^n=\sigma \tau$.    Note also that $$\sigma\tau = \frac{\tau^{\frac{1}{r-1}+r}}{\gamma^\frac{1}{r-1}} = \frac{\tau^s}{\gamma^{\frac{1}{r-1}}}= \frac{1}{6^s \gamma^{\frac{1}{r-1}}} 2^{-ns} = \gamma_1 2^{-n2},$$ where $\gamma_1= \frac{1}{6^s \gamma^\frac{1}{r-1}}$.   By \cite[Proposition $3.2$]{CD},  $\delta^{w^*}_\xi(2^{-n};A^*:(Y^*, |\cdot|_n)\to X^*)\geqslant \sigma \tau= \gamma_1 2^{-ns}$.    

Define the equivalent norm $|\cdot|=\sum_{n=1}^\infty 2^{n(s-p_1)}|\cdot|_n$   on $Y^*$ and note that this is the dual norm to an equivalent norm on $Y$.  Indeed, this norm is induced by the embedding $y^*\mapsto (2^{n(s-p_1)}y^*)_n\in (\oplus_n (Y^*, |\cdot|_n))_{\ell_1}$, which is the adjoint to the surjection $q:(\oplus_n (Y, |\cdot|_n))_{c_0}\to Y$ given by $q((y_n))=\sum_n 2^{n(s-p_1)}y_n$, and the norm $|y|=\inf\{\|(y_n)\|: q((y_n))=y\}$ has $\sum_{n=1}^\infty 2^{n(s-p_1)}|\cdot|_n$ as its dual norm. Let $b\geqslant 1$ be a constant such that for every $n\in \nn$, $b^{-1}|\cdot|\leqslant |\cdot|_n \leqslant b|\cdot|$.

 Fix $0<\ee<1$ and $m\in \nn$ such that $2^{-m}\leqslant \ee< 2^{-m+1}$.  Fix $y^*\in S_{Y^*}$ such that $|y^*|\geqslant 1$, a $B$-tree $T$ with $o(T)=\omega^\xi$, and a $w^*$-null, $(A^*, \ee)$-large collection.  Fix $l> m$ and $\delta>0$.    By passing to a subtree and using $w^*$-nullity, we may assume that for every $t\in B$, $\sup_{t\in B}|y^*+\sum_{s\preceq t} y^*_s|_n\geqslant (1-\delta)|y^*|_n$ for every $1\leqslant n\leqslant l$. Then \begin{align*} \sup_{t\in B} |y^*+\sum_{s\preceq t} y^*_s| & \geqslant \sup_{t\in B}\Bigl[\sum_{n=1}^{m-1}2^{-n(s-p_1)}|y+\sum_{s\preceq t}y_s^*|_n \\ & + 2^{-m(s-p_1)}|y^*+\sum_{s\preceq t}y^*_s|_m + \sum_{n=m+1}^l 2^{-n(s-p_1)}|y^*+\sum_{s\preceq t} y^*_s|_n\Bigr]  \\ & \geqslant (1-\delta)\sum_{m\neq n=1}^l 2^{-n(s-p_1)}|y^*|_n + \sup_{t\in B}2^{-m(s-p_1)}|y^*+\sum_{s\preceq t}y^*_s|_m \\ & \geqslant (1-\delta)\sum_{m\neq n=1}^l 2^{-n(s-p_1)}|y^*|_n + 2^{-m(s-p_1)}|y^*|_m(1+\gamma_1 2^{-ms}) \\ & \geqslant (1-\delta)\sum_{n=1}^l 2^{-n(s-p_1)}|y^*|_n + \gamma_12^{-mp_1}/b \\ & \geqslant (1-\delta)\sum_{n=1}^l 2^{-n(s-p_1)}|y^*|_n+\frac{\gamma_1 \ee^{p_1}}{2^{p_1}b}. \end{align*}  Since $\delta>0$ and $m<l\in\nn$ were arbitrary, $$\sup_{t\in B}|y^*+\sum_{s\preceq t} y^*_s| \geqslant \sum_{n=1}^\infty 2^{-n(s-p_1)}|y^*|_n +\gamma_1\ee^{p_1}/2^{p_1}/b= 1+ \gamma_1\ee^{p_1}/2^{p_1}b.$$   From this we deduce that $\delta^{w^*}_\xi(A^*:(Y^*, |\cdot|)\to X^*)\geqslant \frac{\gamma_1}{2^{p_1}b} \ee^{p_1}$. Therefore $A^*:(Y^*, |\cdot|)\to X^*$ is $w^*$-$\xi$-AUC with power type $p_1$.  By the duality established in \cite{CD}, $A:X\to (Y, |\cdot|)$ is $\xi$-AUS with power type $\frac{p_1}{p_1-1}$.

\end{proof}

\begin{remark} Note that the constant $b$ in the previous proof and the equivalence constant of $|\cdot|$ to the original norm of $Y$ can both be estimated as a function of only $p_1$, while the constant $\gamma$ can be estimated as a function of $\sup_{0<\ee<\|A\|} \ee^p Sz_\xi(A, \ee)$, $\|A\|$, and $p_1$.   

\end{remark}

\begin{proof}[Proof of Corollary \ref{main corollary}]

If $|\cdot|$ is an equivalent norm on $X$ and $0<a\leqslant b<\infty$ are such that $a|\cdot|\leqslant \|\cdot\|\leqslant b|\cdot|$, then for any ordinal $\xi$ and any $0<\sigma<\infty$, $$\rho_\xi(\sigma;I_X:(X, |\cdot|)\to (X, |\cdot|))\leqslant \rho_\xi(\sigma b;I_X:X\to (X, |\cdot|)).$$  Indeed, for any $y\in B_X^{|\cdot|}$, any $0<\sigma<\infty$, any tree $T$ with $o(T)=\omega^\xi$, and any weakly null $(x_t)_{t\in T}\subset B_X^{|\cdot|}$, $(b^{-1}x_t)_{t\in T}\subset B_X$ is weakly null, whence \begin{align*} \inf\{|y+\sigma x|-1: t\in T, x\in \text{co}(x_s:s\preceq t)\}&  = \inf\{|y + \sigma b x|-1: t\in T, x\in \text{co}(x_s/b:s\preceq t)\} \\ & \leqslant \rho_\xi(\sigma b; I_X:X\to (X, |\cdot|)).\end{align*}

The remainder of the proof follows immediately from Theorem \ref{main theorem}.   

\end{proof}

\section{Simultaneity Theorem and applications}
We have another combinatorial result, which is a rather important and useful phenomenon. Indeed, this phenomenon was used twice implicitly in \cite{CD}, and we use it below. Given an ordinal $\xi$, a bounded function $f:\Pi\Omega_\xi \to \rr$, and $\ee\in \rr$, we say $f$ is $\ee$-\emph{large} provided that for any $\delta>0$, there exists an extended pruning $(\theta, e):\Omega_\xi \to \Omega_\xi$ such that for every $(s,t)\in \Pi\Omega_\xi$, $f(\theta(s), e(t))\geqslant \ee-\delta$.

\begin{theorem}[Simultaneity theorem] Fix an ordinal $\xi$, a natural number $n$, and collections $\fff_1, \ldots, \fff_n$ of bounded functions from $\Pi\Omega_\xi$ into $\rr$.  Suppose that $\ee_1, \ldots, \ee_n\in \rr$ are such that for each $1\leqslant i\leqslant n$, no member of $\fff_i$ is $\ee_i$-large.  Then for any $f_1\in \fff_1$, $\ldots$, $f_n\in \fff_n$, $$\inf_{t\in MAX(\Omega_\xi)} \max_{1\leqslant i\leqslant n}\Bigl[\sum_{s\preceq t}\mathbb{P}_\xi(s)f_i(s,t)-\ee_i\Bigr]\leqslant 0.$$   

\label{sim theorem}
\end{theorem}

\begin{proof} Fix $f_1\in \fff_1$, $\ldots$, $f_n\in \fff_n$.  Define $f:\Pi\Omega_\xi\to \rr$ by letting $f(s,t)=\max_{1\leqslant i\leqslant n}[f_i(s,t)-\ee_i]$.  Note that $$\inf_{t\in MAX(\Omega_\xi)} \max_{1\leqslant i\leqslant n}\Bigl[\sum_{s\preceq t}\mathbb{P}_\xi(s)f_i(s,t)-\ee_i\Bigr]\leqslant \inf_{t\in MAX(\Omega_\xi)}\sum_{s\preceq t}\mathbb{P}_\xi(s)f(s,t),$$ so that it suffices to show that the infimum on the right is non-positive.  To obtain a contradiction, assume that $$\inf_{t\in MAX(\Omega_\xi)}\sum_{s\preceq t}\mathbb{P}_\xi(s)f(s,t)=\delta>0.$$   Then by the remarks following Theorem \ref{favorite theorem}, there exists an extended pruning $(\theta, e):\Omega_\xi \to \Omega_\xi$ such that for every $(s,t)\in \Pi\Omega_\xi$, $f(\theta(s), e(t))\geqslant 0$.   Now define $N:\Omega_\xi\to \{1, \ldots, n\} $ by letting $N(s,t)$ denote the minimum $1\leqslant i\leqslant n$ such that $$f(\theta(s), e(t))=f_i(\theta(s), e(t))-\ee_i.$$  By Proposition \ref{recall}$(iii)$, there exists $1\leqslant j\leqslant n$ and an extended pruning $(\theta', e'):\Omega_\xi \to \Omega_\xi$ such that $N(\theta'(\cdot), e'(\cdot))\equiv j$ on $\Pi\Omega_\xi$.   Let $\theta_0=\theta\circ \theta'$ and $e_0=e\circ e'$.   Then for every $(s,t)\in \Pi\Omega_\xi$, $$0\leqslant f(\theta_0(s), e_0(t)) = f_j(\theta_0(s), e_0(t))-\ee_j,$$ from which it follows that $f_j$ is $\ee_j$-large, a contradiction.

\end{proof}

For the next corollaries, for a directed set $D$, we define the $B$-tree $\Gamma_{\xi, \infty}$ by $$\Gamma_{\xi, \infty}=\Bigl\{t_1\cat(\omega^\xi +t_2)\cat \ldots \cat(\omega^\xi(n-1)+t_n): t_i\neq \varnothing, t_i\in \Gamma_\xi, t_i\in MAX(\Gamma_\xi) \text{\ for}1\leqslant i<n\Bigr\}.$$ Given a directed set $D$, we let $\Omega_{\xi, \infty}=\Gamma_{\xi, \infty}D$. We may stratify $\Omega_{\xi, \infty}$ into levels $\Lambda_{\xi, \infty, n}$, $n\in \nn$, as we did with $\Omega_{\xi, N}$.  If $X$ is a Banach space and if $D$ is a weak neighborhood basis at $0$ in $X$, as in the $\Omega_{\xi, n}$ case, a collection $(x_t)_{t\in \Omega_{\xi, \infty}}\subset B_X$ is called normally weakly null if for any $t=t_1\cat (\zeta, U)\in \Omega_{\xi, \infty}$, $x_t\in U$.  

\begin{corollary} Fix an ordinal $\xi$, $1<p\leqslant \infty$, and a norm $1$ operator $A:X\to Y$.  If $1<p\leqslant \infty$ and $A$ is $\xi$-AUS with power type $p$, for any compact subset $C$ of $Y$ a compact subset $T$ of the scalar field, and any $\ee>0$, and any normally weakly null collection $(x_t)_{t\in \Omega_\xi}\subset B_X$, there exists $t\in MAX(\Omega_\xi)$ such that with  $x=\sum_{s\preceq t}\mathbb{P}_\xi(s)x_s\in\text{\emph{co}}(x_s:s\preceq t)$, for any $z\in C$ and any $b\in T$, $$\|z+bA x\|^p\leqslant \|z\|^p+\textbf{\emph{t}}_{\xi, p}(A)b^p+\ee.$$ 

\label{sim corollary}
\end{corollary}

\begin{proof} We perform the proof for $p<\infty$, with the $p=\infty$ case being similar. Of course, we assume $C,T\neq 0$.  For each $z\in C$ and $b\in T$, let $\fff_{z,b}$ consist of all functions $f:\Pi\Omega_\xi\to \rr$ such that there exist a normally weakly null collection $(x_t)_{t\in \Omega_\xi}\subset B_X$ and a collection $(y^*_t)_{t\in MAX(\Omega_\xi)}\subset B_{Y^*}$ such that for every $(s,t)\in \Pi\Omega_\xi$,  $$f(s,t)=\text{Re\ }y^*_t(z+b Ax_s).$$  For any $\ee>0$, no member of $\fff_{z,b}$ is $[\|z\|^p+\textbf{t}_{\xi, p}(A) b^p+\ee/2]^{1/p}$-large. Indeed, if $f\in \fff_{z,b}$ is $[\|z\|^p+\textbf{t}_{\xi, p}(A)b^p+\ee/2]^{1/p}$-large, then there exists a normally weakly null $(x_t)_{t\in \Omega_\xi}\subset B_X$, a collection $(y^*_t)_{t\in MAX(\Omega_\xi)}\subset B_{Y^*}$, and an extended pruning $(\theta, e):\Omega_\xi \to \Omega_\xi$ such that for every $(s,t)\in \Pi\Omega_\xi$, $\text{Re\ }y^*_{e(t)}(z+bAx_{\theta(s)})\geqslant [\|z\|^p+\textbf{t}_{\xi, p}(A) b^p+\ee/4]^{1/p}$.  By relabeling, we may assume $s=\theta(s)$ and $t=e(t)$ for every $s\in \Omega_\xi$ and $t\in MAX(\Omega_\xi)$.   Then it follows that for every $t\in MAX(\Omega_\xi)$, every $x\in \text{co}(x_s:s\preceq t)$, $$\|z+bAx\|^p\geqslant [y^*_t(z+bAx)]^p\geqslant \|z\|^p+\textbf{t}_{\xi, p}(A)b^p+\ee/4.$$  This is a contradiction.

   Let $\delta>0$ be a positive number, let $F\subset C$ be a finite $\delta$-net of $C$, and let $S$ be a finite $\delta$-net of $T$.    Fix $(x_t)_{t\in \Omega_\xi}\subset B_X$ normally weakly null.  For every $z\in F$, $b\in S$, and $t\in MAX(\Omega_\xi)$, fix a functional $y^*_{z,b,t}\in B_{Y^*}$ such that $$\|z+bA\sum_{s\preceq t}\mathbb{P}_\xi(s)x_s\|=y^*_{z,b,t}(Az+bA\sum_{s\preceq t}\mathbb{P}_\xi(s)x_s).$$   Let $f_{z,b}(s,t)=\text{Re\ }y^*_t(z+bAx_s)$.  Then $f_{z,b}\in \fff_{z,b}$, and $f_{z,b}$ is not $[\|z\|^p+\textbf{t}_{\xi, p}(A)b^p+\ee/2]^{1/p}$-large.   Thus by Theorem \ref{sim theorem}, there exists $t\in \Omega_\xi$ such that with $x=\sum_{s\preceq t}\mathbb{P}_\xi(s) x_s$, for every $z\in F$ and $b\in S$, $$\|z+bAx\|^p= \Bigl[\sum_{s\preceq t}\mathbb{P}_\xi(s)f_{z,b}(s,t)\Bigr]^p\leqslant \|z\|^p+\textbf{t}_{\xi, p}(A)b^p+3\ee/4.$$   If $\delta>0$ was chosen small enough, $$\|z+bAx\|^p\leqslant \|z\|^p+\textbf{t}_{\xi, p}(A)b^p +\ee$$ for all $z\in C$ and $b\in T$.

\end{proof}

It is known that, in the case $\xi=0$, if the norm of $X$ is asymptotically uniformly smooth with power type $p$, then every normalized, weakly null tree $(x_t)_{t\in \Omega_{0, \infty}}$ in $X$ admits a branch which is dominated by the $\ell_p$ (resp. $c_0$ if $p=\infty$) basis \cite{JLPS}. This is typically done in the case that $X$ is separable, in which case so is $X^*$, and the set $\Omega_{0, \infty}$ is usually replaced by finite sequences of natural numbers.   In the case that $X^*$ is separable, we may take $D$ to be countable and order isomorphic to $\nn$, recovering the use of that index set.  The next result extends this fact.  

\begin{corollary} Fix an ordinal $\xi$ and $1< p\leqslant \infty$.  If the operator $A:X\to Y$ is $\xi$-AUS with power type $p$, then for any normally weakly null $(x_t)_{t\in \Omega_{\xi, \infty}}\subset B_X$ and for any $C_1>\textbf{\emph{t}}_{\xi, p}(A)$, there exists a sequence $(t_n)_{n=0}^\infty$ such that $\varnothing=t_0\prec t_1\prec t_2\prec \ldots$, for every $n\in \nn$, $t_n\in MAX(\Lambda_{\xi, \infty, n})$,  and with  $z_n=\sum_{t_{n-1}\prec t\preceq t_n}\mathbb{P}_\xi(\iota_{\xi, \infty}(t))x_t\in \text{\emph{co}}(x_t: t_{n-1}\prec t\preceq t_n)$,   $(Az_n)$ is $C_1$-dominated by the $\ell_p$ (resp. $c_0$ if $p=\infty$) basis.   

\label{sim corollary3}
\end{corollary}

\begin{proof}  We treat the $1<p<\infty$ case, with the $p=\infty$ case being similar. Fix $C_1>\textbf{t}_{\xi, p}(A)$.  Fix positive numbers $\ee_1, \ee_2, \ee_3, \ldots$ such that $\sum_{n=1}^\infty \ee_n<\infty$.  Fix a normally weakly null tree $(x_t)_{t\in MAX(\Lambda_{\xi, \infty})}\subset B_X$. Let $t_0=\varnothing$. Let $T=\{b\in \mathbb{K}: |b|\leqslant 1\}$ and let $C_0=\{0\}$.  Fix any $t_1\in MAX(\Lambda_{\xi, \infty, 1})$ such that, with $z_1=\sum_{s\preceq t}\mathbb{P}_\xi(\iota_{\xi, \infty}(s))x_s$, $\|Az_1\|^p\leqslant \textbf{t}_{\xi, p}(A)^p+\ee_1$. We may do this by the definition of $\textbf{t}_{\xi, p}(A)$ applied with $y=0$.   Next, assume that $t_1, \ldots, t_n$ and $z_i=\sum_{t_{i-1}\prec t\preceq t_i}\mathbb{P}_\xi(\iota_{\xi, \infty}(t))x_t\in \text{co}(x_t: t_{i-1}\prec t\preceq t_i)$ have been defined.  Let $C_n=\{A\sum_{i=1}^n a_iz_i:(a_i)_{i=1}^n\in B_{\ell_p^n}\}$.   Note that the proper extensions of $t_n$ which lie on $\Lambda_{\xi, \infty, n+1}$ can be naturally identified with $\Omega_\xi$, so that Corollary \ref{sim corollary} implies the existence of some $t_{n+1}\in MAX(\Lambda_{\xi, \infty, n+1})$ and $z_{n+1}=\sum_{t_n\prec s\preceq t_{n+1}}\mathbb{P}_\xi(\iota_{\xi, \infty}(s))x_s\in \text{co}(x_t: t_n\prec t\preceq t_{n+1})$ such that for every $y\in C_{n+1}$ and $b\in T$, $\|y+bAz_{n+1}\|^p \leqslant \|y\|^p +Cb^p+\ee_{n+1}$.  This completes the recursive construction.  

Now for any $n\in \nn$ and any $(a_i)_{i=1}^n\in S_{\ell_p^n}$, \begin{align*} \|\sum_{i=1}^n a_iAz_i\|^p & \leqslant \|\sum_{i=1}^{n-1}a_iAz_i\|^p+\textbf{t}_{\xi, p}(A)^p|a_n|^p +\ee_n \\ & \leqslant \|\sum_{i=1}^{n-2} a_i Az_i\|^p+ \textbf{t}_{\xi, p}(A)^p(|a_{n-1}|^p+|a_n|^p)+\ee_{n-1}+\ee_n \\ & \leqslant \ldots \leqslant \textbf{t}_{\xi, p}(A)^p\sum_{i=1}^n |a_i|^p + \sum_{i=1}^n\ee_i \\ & < \textbf{t}_{\xi, p}(A)^p+\sum_{i=1}^\infty \ee_i. \end{align*} Since $(\ee_i)_{i=1}^\infty\subset (0, \infty)$ was arbitrary, we are done.

\end{proof}

\begin{rem}\upshape For convenience, let us say that an operator $A$ satisfying the conclusion of the previous corollary \emph{satisfies} $\xi$-$\ell_p$ (resp. $c_0$)-\emph{upper tree estimates with constant} $C$. The inclusion of the constant $C$ in the corollary is unnecessary.  Let us say that an operator $A:X\to Y$ \emph{satisfies} $\xi$-$\ell_p$ (resp. $c_0$)-\emph{upper tree estimates} if for every normally weakly null $(x_t)_{t\in \Omega_{\xi, \infty}}\subset B_X$, there exists for each $n\in \nn$ some $t_n\in MAX(\Lambda_{\xi, \infty, n})$ such that $\varnothing=t_0\prec t_1\prec t_2\prec \ldots$ and, with  $z_n=\sum_{t_{n-1}\prec t\preceq t_n} \mathbb{P}_\xi(\iota_{\xi, \infty}(t))x_t\in \text{co}(x_t: t_{n-1}\prec t\preceq t_n)$, $(Az_n)$ is dominated by the $\ell_p$ (resp. $c_0$) basis.  Here, it is not assumed that there is a constant $C$ so that the image of the convex block is $C$ dominated by the $\ell_p$ or $c_0$ basis.  Standard techniques, however, yield the existence of the constant.  That is, $A$ satisfies $\xi$-$\ell_p$ (resp. $c_0$) upper tree estimates if and only if there exists a $C$ such that $A$ satisfies $\xi$-$\ell_p$ (resp. $c_0$) upper tree estimates with constant $C$.  We only sketch the proof.  Suppose that $B_1, B_2\subset \Omega_{\xi, \infty}$ are well-founded sub-$B$-trees such that $MAX(B_1),MAX(B_2)\subset \cup_{n=1}^\infty MAX(\Lambda_{\xi, \infty, n})$.  Let us define $$B_1+B_2=\{s, t\cat (\omega^\xi k+u):s\in B_1, t\in MAX(B_1)\cap \Lambda_{\xi, \infty, k}, u\in B_2\}.$$    Then $B_1+B_2$ is well-founded and $MAX(B_1+B_2)\subset \cup_{n=1}^\infty MAX(\Lambda_{\xi, \infty, n})$.  Moreover, if $\cup_{n=1}^N \Lambda_{\xi, \infty, n}\subset B_1$ and $\Lambda_{\xi, \infty, 1}\subset B_2$, then $\cup_{n=1}^{N+1}\Lambda_{\xi, \infty, n}\subset B_1+B_2$.  

Similarly, if $B_1, B_2, \ldots\subset \Omega_{\xi, \infty}$ are well-founded sub-$B$-trees such that $\Lambda_{\xi, \infty, 1}\subset B_i$ for all $i\in \nn$, we may let $\sum_{i=1}^n B_i=B_1$ if $n=1$ and $\sum_{i=1}^n B_i= (\sum_{i=1}^{n-1} B_i)+B_n$ if $n>1$.  Then $\sum_{i=1}^n B_i\subset \sum_{i=1}^{n+1}B_i$ for all $n\in \nn$ and $$\bigcup_{n=1}^\infty \sum_{i=1}^n B_i=\Omega_{\xi, \infty}.$$  Finally, suppose that $A:X\to Y$ and $1<p<\infty$ are such that $\|A\|\leqslant 1$ and there is no constant $C$ such that $A$ satisfies $\xi$-$\ell_p$ upper tree estimates.  Let $\partial\Omega_{\xi, \infty}$ denote the infinite sequences all of whose finite initial segments lie in $\Omega_{\xi, \infty}$.   Then for every $n\in \nn$, there exists a normally weakly null collection $(x^n_t)_{t\in \Omega_{\xi, \infty}}\subset B_X$ such that for every $\tau\in \partial \Omega_{\xi, \infty}$, if $(z_i)_{i=1}^\infty$ is the convex block of $(x^n_s)_{s\prec \tau}$ the coefficients of which come from $\mathbb{P}_\xi\circ \iota_{\xi, \infty}$, $(Az_i)_{i=1}^\infty$ fails to be $n$-dominated by the $\ell_p$ basis.  This means there exists a minimal $k_\tau\in \nn$ and $(a_i)_{i=1}^{k_\tau}\in S_{\ell_p^{k_\tau}}$ such that $\|\sum_{i=1}^{k_\tau}a_iAz_i\|>n$.  Of course, $k_\tau>1$, since $\|A\|\leqslant 1$. Let $B_n(\tau)=\{s\in \Omega_{\xi, \infty}: s\prec \tau, s\in \cup_{i=1}^{k_\tau} \Lambda_{\xi, \infty, i}\}$ and $B_n=\cup_{\tau\in \partial\Omega_{\xi, \infty}}B_n(\tau)$.  This defines $B_n$ for each $n\in \nn$.  Note that each $B_n$ is well-founded.  

Next, for any $s\in \Omega_{\xi, \infty}$, there exists a unique $n\in \nn$ such that $s\in \Bigl(\sum_{i=1}^n B_i\Bigr)\setminus \Bigl(\sum_{i=1}^{n-1}B_i\Bigr)$, where  $\sum_{i=1}^0B_i=\varnothing$. In this case, either $n=1$, in which case we set $u=s$, or there exist $k\in \nn$, $t\in MAX(\sum_{i=1}^{n-1} B_i)$, and $u\in B_n$ such that $s=t\cat (\omega^\xi k+u)$. We let $x_s=x^n_u$.   This defines a normally weakly null collection $(x_s)_{s\in \Omega_{\xi, \infty}}\subset B_X$.  Furthermore, for any $\tau\in \partial\Omega_{\xi, \infty}$, if $(z_i)_{i=1}^\infty$ is the special convex block of $(x_s:s\prec \tau)$,  there exist $0=k_0<k_1<\ldots$ such that for all $n\in \nn$, $(z_i)_{i=k_{n-1}+1}^{k_n}$ is a convex block of a branch of $(x^n_t)_{t\in \Omega_{\xi, \infty}} $ which is not $n$-dominated by the $\ell_p$ basis. Therefore the collection $(x_s)_{s\in \Omega_{\xi, \infty}}$ witnesses the fact that $A$ does not satisfy $\xi$-$\ell_p$ (resp. $\xi$-$c_0$) upper tree estimates.

\end{rem}

\begin{rem}\upshape

Note that there is a partial converse to Corollary \ref{sim corollary3}.  Suppose that $A$ satisfies $\xi$-$\ell_p$-upper tree estimates with constant $C$. Then for any $k\in \nn$, $B$-tree $B$ with $o(B)=\omega^\xi k$, $\ee>0$, and any weakly null collection $(x_t)_{t\in B}\subset B_X$ such that for all $t\in B$ and $x\in \text{co}(x_s:s\preceq t)$, $\|Ax\|\geqslant \ee$, $\ee\leqslant C/k^{1/p'}$, and we may appeal to Corollary \ref{illinoiscorollary} to deduce that $\textbf{p}_\xi(A)\leqslant p'$.  Indeed, we first note that by standard techniques given in \cite{CD}, we may first fix $\theta:\Omega_{\xi, k}\to B$ such that $(u_t)_{t\in \Omega_{\xi, k}}:=(x_{\theta(t)})_{t\in \Omega_{\xi, k}}$ is normally weakly null and then use this collection to fill out the first $k$ levels of a normally weakly null collection $(z_t)_{t\in \Omega_{\xi, \infty}}$, the rest of which is filled by zeros.  Then there exists a branch and $(z_i)_{i=1}^k$ a convex block of a branch of the $(u_t)_{t\in \Omega{\xi, k}}$, and therefore the $(x_t)_{t\in B}$ vectors, the image under $A$ of which is $C$-dominated by the $\ell_p$ basis.  Thus $$\|Ak^{-1}\sum_{i=1}^k z_i\|\leqslant Ck^{1/p}/k= C/k^{1/p'}.$$  

\label{remark1}
\end{rem}

\begin{rem}\upshape In the previous corollary and remark, $D$ was a weak neighborhood basis at $0$ in $X$.  If we take $D$ to be a $w^*$-neighborhood basis at $0$ in $Y^*$, we may use dualization results from \cite{CD} to deduce that if $A^*$ is $w^*$-$\xi$-AUC with power type $q$, there exists a constant $C'$ such that for every normally $w^*$-null, $(A^*, 1)$-large collection $(y^*_t)_{t\in \Omega_{\xi, \infty}}$, there exist $\varnothing =t_0\prec t_1\prec \ldots$ such that for each $n\in \nn$, $t_n\in MAX(\Lambda_{\xi, \infty, n})$ and such that, with $z^*_n=\sum_{t_{n-1}\prec t\preceq t_n}y^*_t$, $(z^*_n)$ $C$-dominates the $\ell_{p'}$ basis. Let us say in this case that $A^*$ satisfies $w^*$-$\xi$-$\ell_q$ \emph{lower tree estimates} (\emph{with constant} $C$, if we are concerned with the constant). One may also deduce by this fact that $\textbf{p}_\xi(A)\leqslant p'$, using Corollary \ref{illinoiscorollary}$(ii)$.   We summarize this discussion with the previous renorming theorems to deduce the following, which is the higher ordinal version of quantitative results from \cite{DK} characterizing the non-operator, $\xi=0$ case.

\label{remark2}
\end{rem}

\begin{theorem} Fix an ordinal $\xi$, an operator $A:X\to Y$, and $1<p<\infty$.  The following are equivalent. \begin{enumerate}[(i)]\item $\textbf{\emph{p}}_\xi(A)\leqslant p'$. \item For every $1<r<p$, $A$ satisfies $\xi$-$\ell_r$ upper tree estimates.  \item For every $p'<r<\infty$, $A^*$ satisfies $w^*$-$\xi$-$\ell_r$ lower tree estimates. \item For every $1<r<p$, $A$ admits a $\xi$-AUS norm with power type $r$. \item For every $p'<r<\infty$, $A$ admits a norm making $A^*$ $w^*$-$\xi$-AUC with power type $r$.   \end{enumerate}

\label{cross}

\end{theorem}

We recall the following shown in \cite{CD}, the proof of which implicitly uses Theorem \ref{sim theorem}.  

\begin{corollary} Fix an ordinal $\xi$ and $1<p\leqslant \infty$. Then for any $\sigma>0$, $2\rho_\xi(\sigma;A+B)\leqslant \rho_\xi(2\sigma;A)+\rho_\xi(2\sigma;B)$.  Consequently, if $A,B:X\to Y$ are $\xi$-AUS with power type $p$, so is $A+B$.  

\label{sim corollary 2}
\end{corollary}

We have one more application of Theorem \ref{sim theorem}.

\begin{corollary} Fix an ordinal $\xi$.  Let $\Lambda$ be a non-empty set and suppose that for every $\lambda\in \Lambda$, $A_\lambda:X_\lambda\to Y_\lambda$ is an operator with $Sz(A)\leqslant \omega^{\xi+1}$. Suppose also that $\sup_\lambda \|A_\lambda\|<\infty$.  \begin{enumerate}[(i)]\item For any $1<p\leqslant q< \infty$, the operator $\oplus_\lambda A_\lambda: (\oplus_{\lambda\in \Lambda} X_\lambda)_{\ell_q(\Lambda)}\to (\oplus_{\lambda \in \Lambda}Y_\lambda)_{\ell_q(\Lambda)}$ is $\xi$-AUS with power type $p$ if and only if $\sup_{\lambda\in \Lambda} \sup_{\sigma\in (0,1)} \frac{\rho_\xi(\sigma;A_\lambda)}{\sigma^p}$ if and only if $\sup_{\lambda\in \Lambda}\textbf{\emph{t}}_{\xi,p}(A_\lambda)<\infty$. More precisely, $\textbf{\emph{t}}_{\xi, p}(\oplus A_\lambda)\leqslant \sup_{\lambda\in \Lambda} \max\{\|A_\lambda\|^p, \textbf{\emph{t}}_{\xi, p}(A_\lambda)\}$.  \item $\oplus A_\lambda:(\oplus_{\lambda\in \Lambda}X_\lambda)_{c_0(\Lambda)}\to (\oplus_{\lambda\in \Lambda}Y_\lambda)_{c_0(\Lambda)}$ is $\xi$-AUS with power type $\infty$ if and only if $\sup_{\lambda\in \Lambda}\textbf{\emph{t}}_{\xi, \infty}(A_\lambda)<\infty$, and in this case, $\textbf{\emph{t}}_{\xi, \infty}(\oplus A_\lambda)\leqslant \sup_{\lambda\in \Lambda}\max\{\|A_\lambda\|, \textbf{\emph{t}}_{\xi, \infty}(A_\lambda)\}$. 

\end{enumerate}   

\label{narcor}

\end{corollary}

\begin{proof} By Proposition \ref{mymaxinfo}, $\sup_{\lambda\in \Lambda} \sup_{0<\sigma<1}\rho_\xi(\sigma;A_\lambda)/\sigma^p<\infty$ if and only if $\sup_{\lambda\in \Lambda}\textbf{t}_{\xi, p}(A_\lambda)<\infty$.  Let $X=(\oplus_{\lambda\in \Lambda}X_\lambda)_{\ell_q(\Lambda)}$, $Y=(\oplus_{\lambda\in \Lambda}Y_\lambda)_{\ell_q(\Lambda)}$, and $A=\oplus A_\lambda$. If $A$ is $\xi$-AUS with power type $p$, then clearly $\sup_{\lambda\in \Lambda}\sup_{\sigma\in (0,1)}\frac{\rho_\xi(\sigma;A_\lambda)}{\sigma^p}<\infty$, since $\sup_{\sigma\in (0,1)}\frac{\rho_\xi(\sigma;A)}{\sigma^p}<\infty$ and $\rho_\xi(\sigma;A_\lambda)\leqslant \rho_\xi(\sigma;A)$ for any $\sigma\in (0,1)$ and $\lambda\in \Lambda$.   

Suppose $\sup_{\lambda\in \Lambda}\sup_{\sigma\in (0,1)}\frac{\rho_\xi(\sigma;A_\lambda)}{\sigma^p}<\infty$. For each $\lambda_0\in \Lambda$, let $P_{\lambda_0}:(\oplus_{\lambda\in \Lambda}X_\lambda)_{\ell_q(\Lambda)}\to X_{\lambda_0}$, $Q_{\lambda_0}:(\oplus_{\lambda\in \Lambda}Y_\lambda)\to Y_{\lambda_0}$ be the projections.  It is quite clear that for any $\lambda\in \Lambda$, $\rho_\xi(\cdot;Q_\lambda AP_\lambda:X\to Y_\lambda)=\rho_\xi(\cdot;A_\lambda)$. Then $C=\sup_{\lambda\in \Lambda}\sup \max\{\|A_\lambda\|^p, \textbf{t}_{\xi, p}(A_\lambda)\}$ is finite.  Fix $y\in (\oplus_\lambda Y_\lambda)_{\ell_q(\Lambda)}$ and a normally weakly null $(x_t)_{t\in \Omega_\xi}$. We may assume that $F:=\{\lambda\in \Lambda: Q_\lambda y\neq 0\}$ is finite, by the density of such members of $Y$.     Fix $\ee>0$.   By \cite[Lemma $4.2$]{CD}, we may pass to a subtree and assume that there exist non-negative scalars $a$ and $a_\lambda$, $\lambda\in F$, such that for any $t\in \Omega_\xi$, $\|P_\lambda x\|\leqslant a_\lambda$, $\|x_t-P_Fx_t\|\leqslant a$, and $a^p+\sum_{\lambda\in F}a_\lambda^p\leqslant \sigma^p+\ee/2$.   Using Theorem \ref{sim theorem} as before applied to $Q_\lambda AP_\lambda$, we may fix $t\in MAX(\Omega_\xi)$ and $z\in \text{co}(x_s:s\preceq t)$ such that for every $\lambda\in F$, $$\|Q_\lambda(y+Az)\|^p \leqslant \|Q_\lambda y\|^p + Ca_\lambda^p+\ee/2|F|.$$   Since $p\leqslant q$,   \begin{align*} \|y+Az\|^p & \leqslant \sum_{\lambda\in F}\|Q_\lambda(y+Az)\|^p+ \|Q_{\Lambda\setminus F}Az\|^p  \\ & \leqslant \|y\|^p+C\sum_{\lambda\in F}a_\lambda^p+Ca^p+\ee/2 \leqslant \ee + \|y\|^p+C\sigma^p.  \end{align*} This implies that $A$ is $\xi$-AUS with power type $p$.

\end{proof}

\section{Ideal properties}

In this section, we make some observations about sums and factorization properties of the $\xi$-Szlenk power type of an operator.  In what follows, $\oplus$ denotes the \emph{Hessenberg} sum of two ordinals.  A full introduction of the Hessenberg sum would be unnecessarily technical, so we only indicate the only properties we will need here: Given an ordinal $\xi$ and natural numbers $m,n$, $\omega^\xi m \oplus \omega^\xi n=\omega^\xi (m+n)$, and if $\zeta\leqslant \zeta_1$, $\xi\leqslant \xi_1$, $\zeta\oplus \xi\leqslant \zeta_1\oplus \xi_1$.

The following results can be found in \cite{BrookerAsplund} and \cite{CIllinois}.  

\begin{proposition} There exists a constant $c>0$ such that the following hold.  \begin{enumerate}[(i)]\item If $A:Y\to Z$, $B:X\to Y$, $C:W\to X$ are non-zero operators, then for any $\ee>0$, $Sz(ABC, \ee)\leqslant Sz(B, \|A\|\|C\|c\ee)$, \item if $A:X\to Y$ is an operator, $q:W\to X$ is a quotient map, and $j:Y\to Z$ is an isometric embedding, $Sz(A, \ee)\leqslant Sz(jAq, c\ee)$. \item For any $A,B:X\to Y$ and $\ee>0$, $Sz(A+B, \ee)\leqslant Sz(A, c\ee)\oplus Sz(B, c\ee)$. In particular, for any ordinal $\xi$,  $\textbf{\emph{p}}_\xi(A+B)\leqslant \textbf{\emph{p}}_\xi(A)+\textbf{\emph{p}}_\xi(B)$.  \end{enumerate}

\end{proposition}

\begin{corollary} Let $\xi$ be any ordinal.  \begin{enumerate}[(i)]\item For any $A:Y\to Z$, $B:X\to Y$, and $C:W\to X$, $\textbf{\emph{p}}_\xi(ABC)\leqslant \textbf{\emph{p}}_\xi(B)$. \item If $A:X\to Y$ is any operator, $q:W\to X$ is a quotient map, and $j:Y\to Z$ is an isometric embedding, $\textbf{\emph{p}}_\xi(jAq)=\textbf{\emph{p}}_\xi(A)$.  \item For any two operators $A,B:X\to Y$, $\textbf{\emph{p}}_\xi(A+B) \leqslant \max\{\textbf{\emph{p}}_\xi(A), \textbf{\emph{p}}_\xi(B)\}$. \item For any ordinal $\xi$ and $1\leqslant p<\infty$, the class of operators $A$ such that $\textbf{\emph{p}}_\xi(A)\leqslant p$ is an injective, surjective operator ideal.  \end{enumerate}

\end{corollary}

\begin{lemma} Fix $1<p<\infty$, two sequences $(X_n)$, $(Y_n)$ of Banach spaces,  and an operator $A:X:=(\oplus X_n)_{\ell_p}\to Y:=(\oplus_n Y_n)_{\ell_p}$. For each $m,n\in \nn$, let $A_{mn}=Q_nAP_m$, where $P_m:X\to X_m$, $Q_n:Y\to Y_n$ are the canonical projections. Then if $\xi$ is an ordinal such that for every $m,n\in \nn$, $Sz(A_{mn})\leqslant \omega^\xi$, then $\textbf{\emph{p}}_\xi(A)\leqslant p'$.    

\label{stephanie}
\end{lemma}

\begin{proof} If $Sz(A)=\omega^\xi$, then $\textbf{p}_\xi(A)=0$.   Thus we may assume $Sz(A)>\omega^\xi$.  It follows from \cite{BrookerAsplund} that $Sz(A)\leqslant \omega^{\xi+1}$.  In \cite{BrookerAsplund}, Brooker showed that for any ordinal $\xi$, the class of operators with Szlenk index not exceeding $\omega^\xi$ is closed under finite sums.  Thus for any $m,n\in \nn$, $Q_{[1,n]}AP_{[1,m]}:=\sum_{i=1}^m\sum_{j=1}^n A_{ij}$ has Szlenk index not exceeding $\omega^\xi$. Using Theorem \ref{illinoistheorem}, it follows that for any $m,n\in \nn$, any $B$-tree $B$ with $o(B)=\omega^\xi$, any normally weakly null $(x_t)_{t\in B}\subset B_X$, there exists $t\in B$ and $x\in \text{co}(x_s:s\preceq t)$ such that $\|Q_{[1,n]}AP_{[1,m]}x\|<\ee$.   

Fix $\ee>0$ and suppose $k\in \nn$ is such that there exists a $B$-tree $B$ with $o(B)=\omega^\xi k$ and a weakly null collection $(x_t)_{t\in B}\subset B_X$.  Fix $\delta>0$ and recursively select $\varnothing=t_0\prec t_1\prec \ldots \prec t_k$, $z_i\in \text{co}(x_t: t_{i-1}\prec t\preceq t_i)$, $m_0=n_0=0$, and $m_1,n_1, \ldots, m_k, n_k$ such that for each $1\leqslant i\leqslant k$,  \begin{enumerate}[(i)]\item $t_i\in B^{\omega^\xi(k-i)}$, \item $m_{i-1}<m_i$, \item  $n_{i-1}<n_i$,  \item $\|P_{(m_i, \infty)}z_i\|<\delta$, \item $\|Q_{(n_i, \infty)}AP_{[1, m_{i-1}]}z_i\|<\delta$, \item $\|Q_{[1, n_{i-1}]}AP_{[1, m_{i-1}]}z_i\|<\delta$.  \end{enumerate}  We indicate how to make these choices.  Assume that for some $0\leqslant i<n$ and each $0\leqslant j\leqslant i$,  $t_j$, $m_j$, $n_j$ have been chosen to have the properties above. Let $U$ denote the non-empty sequences $u$ such that $t_i\cat u\in B^{\omega^\xi(k-i-1)}$, then $(x_{t_i\cat u})_{u\in U}\subset B_X$ is a weakly null collection.  By the previous paragraph, there exists $u\in U$ and $z_{i+1}\in \text{co}(x_{t_i\cat t}: t\preceq u)$ such that $\|Q_{[1, n_i]}AP_{[1, m_i]}z_{i+1}\|<\delta$.   We set $t_{i+1}=t_i\cat u\in B^{\omega^\xi(k-i-1)}$ and then choose $m_{i+1}$, $n_{i+1}$ to satisfy (ii)-(v).       

Let $z=\frac{1}{k}\sum_{i=1}^k z_i\in \text{co}(x_t:t\preceq t_k)$.   Let $$u_i=P_{[1, m_{i-1}]} z_i, \hspace{5mm} v_i=P_{(m_{i-1}, m_i]}z_i, \hspace{5mm} w_i=P_{(m_i, \infty)},$$ and let $$u_i'=Q_{[1, n_{i-1}]}Au_i,\hspace{5mm} u_i''=Q_{(n_{i-1}, n_i]}Au_i, \hspace{5mm} u_i'''=Q_{(n_i, \infty)}Au_i.$$   Note that $Az_i=Av_i+Aw_i+u_i'+u_i''+u_i'''$, $\|Aw_i\|\leqslant \|A\|\delta$, $\|u_i'\|, \|u_i'''\|<\delta$, $$\|A\sum_{i=1}^k v_i\|\leqslant \|A\|\bigl(\sum_{i=1}^k \|v_i\|^p\bigr)^{1/p}/k\leqslant \|A\|/k^{1/p'},$$ and $$\|\sum_{i=1}^k u_i''\|\leqslant \bigl(\sum_{i=1}^k \|u_i''\|^p\bigr)^{1/p} \leqslant \|A\|/k^{1/p'}.$$  Therefore $\|Az\|\leqslant 3\delta+2\|A\|/k^{1/p'}$, and $$\inf\{\|Ax\|: t\in B, x\in \text{co}(x_s:s\preceq t)\}\leqslant \frac{2\|A\|}{k^{1/p'}}.$$   Thus $\ee\leqslant 2\|A\|/k^{1/p'}$.   We now finish by Corollary \ref{illinoiscorollary}.

\end{proof}

We next recall a result of Brooker.  

\begin{theorem}\cite[Theorem $4.1$]{BrookerAsplund} If $A:X\to Y$ is an operator with $Sz(A)=\omega^\xi$, then $A$ factors through a Banach space $Z$ with $Sz(Z)\leqslant \omega^{\xi+1}$. 
\label{brookertheorem}
\end{theorem}

We next include an observation regarding Brooker's result.  The Szlenk index of an operator is a method for quantifying ``how'' Asplund that operator is. Among those operators with Szlenk index not exceeding $\omega^{\xi+1}$, the quantity $\textbf{p}_\xi(\cdot)$ allows further quantification of ``how'' Asplund an operator is.  The lower the $\textbf{p}_\xi(\cdot)$ value, the ``more'' Asplund the operator is.  In this sense, the next corollary provides more quantification of ``how'' Asplund a Banach space through which a given Asplund operator can be.  

The next result is an observation we make regarding a result of Brooker \cite{BrookerAsplund}, who used interpolation techniques due to Heinrich \cite{Heinrich}.  Brooker proved that any operator $A:X\to Y$ with $Sz(A)=\omega^\xi$ factors through a Banach space $Z$ with $Sz(Z)\leqslant \omega^{\xi+1}$.  Moreover, this space $Z$ was taken to be an interpolation space.  We refer the reader to the works of Brooker and Heinrich for the relevant definitions and background regarding interpolation.

\begin{corollary} For every ordinal $\xi$, every operator $A:X\to Y$ with $Sz(A)\leqslant \omega^\xi$, and every $1<p<\infty$, there exists a Banach space $Z$ with $\textbf{\emph{p}}_\xi(Z)\leqslant p$ such that $A$ factors through $Z$.

\end{corollary}

\begin{proof}

By first passing to the quotient $\overline{A}:X/\ker(A)\to Y$ given by $\overline{A}(x+\ker(A))=Ax$, we may assume $A$ is injective.  For any $1<p<\infty$, there exist sequences $(X_n)$, $(Y_n)$ of Banach spaces, a Banach space $G$, a quotient map $q:(\oplus_n X_n)_{\ell_p}\to G$, and an isometric embedding $j:G\to (\oplus_n Y_n)_{\ell_p}$, and an operator $B:(\oplus_n X_n)_{\ell_p}\to (\oplus_n Y_n)_{\ell_p}$ such that for every $m,n\in \nn$, $Q_n BP_m:X_m\to Y_n$ is simply the operator $A:X\to Y$ with equivalent norms on the domain and range, and $B=jq$.  Here, $G$ is an interpolation space of $X$ and $Y$, where $X$ is continuously mapped into $Y$ by $A$. We deduce by Lemma \ref{stephanie} that $\textbf{p}_\xi(B)\leqslant p'$.  Since $B=jq=j\text{Id}_Gq$, by injectivity and surjectivity, we deduce that $\textbf{p}_\xi(\text{Id}_G)\leqslant p'$.

\end{proof}

\section{Closing remarks}

 The submultiplicative nature of the $\ee$-Szlenk index of a Banach space, which was shown by Lancien in \cite{Lancien}, yields that if $Sz(X)=\omega$, then the Szlenk power type $\textbf{p}(X)=\textbf{p}_0(X)$ is finite. Therefore in the case of a Banach space having Szlenk index $\omega$, an equivalent norm with power type AUS modulus is automatic.  However, it is easy to construct operators with Szlenk index $\omega$ which admit no equivalent norm with power type AUS modulus. Indeed, one can easily construct an example $\oplus a_nI_{\ell_{p_n}}:(\oplus \ell_{p_n})_{\ell_2}\to (\oplus \ell_{p_n})_{\ell_2}$, where $a_n\downarrow 0$ slowly and $p_n\downarrow 1$ rapidly,  has Szlenk index $\omega$ but which cannot be renormed to be AUS with any power type.   This possible because the $\ee$-Szlenk index of an operator need not be submultiplicative.  

Moreover, if $\xi>0$, then for some $\ee>0$, $\omega^\xi<Sz(A, \ee)<\omega^{\xi+1}$, and $Sz(A,\ee)^2>\omega^{\xi 2}\geqslant \omega^{\xi+1}$.  Thus the submultiplicative nature of the Szlenk index of a Banach space only gives a non-trivial application to the Szlenk power type of a Banach space in the case $Sz(X)=\omega$. Indeed,  for every ordinal $\xi>0$, any $1> \theta_1>\theta_2>\ldots$ with $\lim_n \theta_n=0$, and any natural numbers $k_1<k_2<\ldots$, a construction from \cite{C} was carried out which yields a Banach space $X$ with $Sz(X)=\omega^{\xi+1}$ such that for each $n\in \nn$, $Sz(X, \theta_n) >\omega^\xi k_n$.  From this it easily follows that this space can be taken to have $\textbf{p}_\xi(X)=\infty$.  Indeed, this is the case if we take $k_n=n$ and $\theta_n=1/\log_2(n+2)$.

  Furthermore, for any $1<p\leqslant \infty$ and every ordinal $\xi$, an example was given in \cite{CD} of a Banach space which is $\xi$-AUS with power type $p$ and which cannot be renormed to be $\xi$-AUS with any power type better than $p$. Indeed, if $1<p<\infty$, we let $S_{0, p}= \ell_p$. If $S_{\xi,p}$ has been defined, we let $$S_{\xi+1, p}= (\oplus_{n=1}^\infty \ell_1^n(S_{\xi,p}))_{\ell_p}.$$  Finally, if $\xi$ is a limit ordinal and $S_\zeta$ has been defined for every $\zeta<\xi$, we let $$S_{\xi,p}= (\oplus_{\zeta<\xi} S_{\zeta, p})_{\ell_p([0, \xi))}.$$  Replacing $\ell_p$ with $c_0$ and $\ell_p$ direct sums with $c_0$ direct sums gives $S_{\xi,\infty}$.    An easy proof by induction yields that for any $1<p< \infty$ and any ordinal $\xi$, there exist a sequence $X_{\xi, p,n}$ of subspaces of $S_{\xi,p}$ such that $S_{\xi,p}=(\oplus_n X_{\xi, p,n})_{\ell_p}$ and for each $n\in \nn$, there exists a weakly null tree $(x_t)_{t\in T}\subset S_{X_{\xi,p,n}}$ of order $\omega^\xi$ such that $\|x\|=1$ for every $t\in T$ and every $x\in \text{co}(x_s: s\leqslant t)$.   Indeed, $S_{0,p}=\ell_p=\ell_p(\ell_p)$, if $M_1, M_2, \ldots$ are infinite, disjoint subsets of $\nn$, $$S_{\xi+1, p}= (\oplus_{n=1}^\infty \ell_1^n(S_{\xi,p}))_{\ell_p}= (\oplus_{n=1}^\infty (\oplus_{m\in M_n} \ell_p^n(S_{\xi,p}))_{\ell_p}).$$  Finally, if $\xi$ is a limit ordinal, for each $n\in \nn$, let $M_n$ denote the set of $\zeta<\xi$ such that $\zeta= \lambda+n-1$ where $\lambda$ is either zero or a limit ordinal.  Then with $X_{\xi,p,n}=(\oplus_{\zeta\in M_n} S_{\zeta, p})_{\ell_p(M_n)}$, $S_{\xi,p}=(\oplus_{n=1}^\infty X_{\xi,p,n})_{\ell_p}$.

The fact in the previous proof shows that $Sz(S_{\xi,p})\geqslant \omega^{\xi+1}$ and that $S_{\xi,p}$ cannot have $\textbf{p}_\xi(S_{\xi,p})=r>p$.   Indeed, for any $n\in \nn$, we built a weakly null tree $(x_t)_{t\in \Omega_{\xi,n}}$ by populating the first level using the weakly null tree from the previous paragraph in $X_{\xi, p,1}$, the second level is populated by the tree from the previous paragraph in $X_{\xi,p,2}$, etc.   Then for any $t\in MAX(\Omega_{\xi,n})$, $$\|\sum_{i=1}^n \sum_{\Lambda_{\xi,n,i}\ni s} \mathbb{P}_\xi(\iota_{\xi,n}(s)) x_s\|^p= n.$$   Finally, one can show by induction that $Sz(S_{\xi,p})\leqslant \omega^{\xi+1}$ and, in fact, $\textbf{t}_{\xi,p}(S_{\xi,p})\leqslant 1$. Note that the latter property implies the former. The base case of the induction is clear, while the second follows from Corollary \ref{narcor}, once we remark that, since for $\zeta<\xi$ and $n\in \nn$,  $Sz(S_{\zeta, p}), Sz(\ell_1^n(S_{\zeta, p}))\leqslant \omega^\xi$, whence $$\textbf{t}_{\xi,p}(S_{\zeta, p})=\textbf{t}_{\xi, p}(\ell_1^n(S_{\zeta, p}))=0.$$

	For each ordinal $\xi$ and $p\in \{0\}\cup [1, \infty)$,  let $\textbf{Pz}_{\xi, p}$ denote the class of all operators $A$ such that $\textbf{p}_\xi(A)\leqslant p $.  We summarize what we have shown in the following.  

\begin{theorem} Fix an ordinal $\xi$ and $p\in \{0\}\cup [1, \infty)$.   \begin{enumerate}[(i)]\item For any ordinal $\zeta$ and $q\in \{0\}\cup [1, \infty)$, $\textbf{\emph{Pz}}_{\xi,p}\subset \textbf{\emph{Pz}}_{\zeta, q}$ if and only if $\xi\leqslant \zeta$ or $\zeta=\xi$ and $p\leqslant q$.  In particular, these classes are all distinct.  \item $\textbf{\emph{Pz}}_{\xi, p}$ is an injective, surjective operator ideal. \item  $\textbf{\emph{Pz}}_{\xi, 0}$ is the class of all operators with Szlenk index not exceeding $\omega^\xi$.   \end{enumerate}

\end{theorem}

We last discuss the sharpness of the main renorming theorem.  The optimal result would be to deduce that if $A:X\to Y$ is an operator such that $\textbf{p}_\xi(A)\leqslant p\in [1, \infty)$, then there exists an equivalent norm on $Y$ making $A$ $\xi$-AUS with power type $p'$.  However, the main renorming theorem only guarantees that for each $1<r<p'$, there exists an equivalent norm on $Y$ making $A$ $\xi$-AUS with power type $r$.  In general, this is optimal.  Indeed, as was noted in \cite{KOS}, Tsirelson's space \cite{Tsirelson} $T$ has $\textbf{p}_0(T)=\infty$, but it cannot be renormed to be AUS with power type $\infty$.   Moreover, if for $1<p<\infty$ and $q$ such that $1/p+1/q=1$, $T^*_q$ denotes the $q$-convexification of $T^*$, which is the Figiel-Johnson Tsirelson space \cite{FJ},  the basis $(e^*_n)$ of $T^*_q$ has the property that every normalized block of it dominates every $\ell_r$ basis for $q<r$. Let $E=(T^*_q)^*$.   Then $\textbf{p}_0(E)=q$, but $E$ cannot be renormed to be AUS with power type $p$.  This is because the basis of $E$ is normalized and weakly null, but has no subsequence dominated by the $\ell_p$ basis.  More generally, if $\xi$ is a limit ordinal with countable cofinality and if $\xi_n\uparrow \xi$,  we may fix a sequence of spaces $X_n$ such that $Sz(X_n, 1)>\omega^{\xi_n}$ and $Sz(X_n)=\omega^{\xi_n+1}$ (see for example \cite{BrookerAsplund} or \cite{CD}).  Then the direct sum $X=(\oplus_n X_n)_E$ has Szlenk index $\omega^{\xi+1}$, $\textbf{p}_\xi(X)=p$, but $X$ cannot be renormed to be $\xi$ AUS with power type $q$.  Let us sketch how to prove these assertions.   Fix a normally weakly null $(x_t)_{t\in \Omega_{\xi, \infty}}\subset B_X$.  Arguing as in Lemma \ref{stephanie}, we may recursively select $0=n_0<n_1<\ldots$, $\ee_i>0$,  and $z_i\in \text{co}(x_s:t_{i-1}\prec s\preceq t_i)$ such that $\|z_i-P_{\oplus_{j=1}^{n_i} X_j}z_i\|$, $\|P_{\oplus_{j=1}^{n_{i-1}} X_j} z_i\|<\ee_i$.  If $\pi:X\to E$ is given by $\pi((x_n))=\sum \|x_n\|f_n$, where $(f_n)$ is the basis of $E$, then $(\pi(z_i))$ will be $2$-equivalent to $(z_i)$ and $2$-equivalent to a block sequence in $B_E$, given an appropriate choice of $(\ee_i)$.  Since every normalized block of $E$ is dominated by the $\ell_r$ basis for $1<r<q$, we deduce that for every $1<r<q$, $X$ satisfies $\xi$-$\ell_r$ upper tree estimates.  By Theorem \ref{cross}, we deduce that $\textbf{p}_\xi(X)\leqslant p$.    Next, since $Sz(X_n, 1)>\omega^{\xi_n}$, by Theorem \ref{illinoistheorem}, there exists $\ee>0$ such that for any $n\in \nn$, there exists a $B$-tree $B$ with $o(B)=\omega^{\xi_n}$ and a collection $(x^n_t)_{t\in B_n}\subset B_{X_n}$ such that for every $t\in B_n$ and $x\in \text{co}(x_s:s\preceq t)$, $\|x\|\geqslant \ee$. By renaming the index sets, we may assume the $B_n$ are disjoint sets and let $B=\cup_{n=1}^\infty B_n$.  For $t\in B$, we may let $x_t=x_t^n$, where $t\in B_n$.   We construct a normally weakly null collection $(u_t)_{t\in \Omega_{\xi, \infty}}\subset B_X$ such that for every $\tau\in \partial \Omega_{\xi, \infty}$, if $\varnothing=t_0\prec t_1\prec \ldots$, $t_i\in MAX(\Lambda_{\xi, \infty, i})$, then there exist $k_1<k_2<\ldots$ such that $(u_s:t_{i-1}\prec s\preceq t_i)\subset X_{k_i}$ and $\|x\|\geqslant \ee$ for all $x\in \text{co}(u_s: t_{i-1}\prec s\preceq t_i)$.    This will yield that every convex block $(z_i)$ of a branch of $(u_t)_{t\in \Omega_{\xi, \infty}}$ is equivalent to a seminormalized block if $E$, and is therefore not dominated by the $\ell_q$ basis. This collection will show that $X$ does not satisfy $\xi$-$\ell_q$-upper tree estimates, and therefore that $X$ cannot be renormed to be $\xi$-AUS with power type $q$.    The choice of $(u_t)_{t\in \Omega_{\xi, \infty}}$ is done by choosing $(u_t)_{t\in \Lambda_{\xi, \infty, i}}$ by induction on $i$.  The base and successor cases are similar, so we only perform the successor case.   Fix $t\in MAX(\Lambda_{\xi, \infty, i})$ and assume that $k_1<\ldots <k_i$ and $x_s$, $s\preceq t$, are already chosen to have the indicated properties. Let $U$ denote those members of $\Lambda_{\xi, \infty, i+1}$ which are proper extensions of $t$ (noting that this set may be identified with $\Omega_\xi$) and let $C$ be the disjoint union of $\cup_{j=k_i+1}^\infty B_j$.  Note that $o(C)=\omega^\xi$, so we may fix a monotone function $\theta:U\to B$ such that, $(x_{\theta(t)})_{t\in U}$ is normally weakly null.  Let $u_t=x_{\theta(t)}$.   This completes the construction, and the indicated properties are easily verified.

A similar construction may be undertaken if $\xi=\zeta+1$ to obtain a Banach space $X$ with $\textbf{p}_\xi(X)=p$ but such that $X$ cannot be renormed to be $\xi$-AUS with power type $p$.  Indeed, we may find some space $Y$ with $Sz(Y,1)>\omega^\zeta$ and $Sz(Y)=\omega^{\zeta+1}$ and let $X=(\oplus_n \ell_1^n(Y))_E$.

\end{document}